\documentclass[10pt,reqno]{amsart}

\usepackage[a4paper,left=27mm,right=27mm,top=25mm,
bottom=25mm,marginpar=25mm]{geometry}

\usepackage{amsmath, amssymb ,amsthm, amsfonts, amsgen,
mathrsfs}

\usepackage{pvscript}
\usepackage{frcursive}

\usepackage{mathrsfs}

\usepackage{amsmath}
\usepackage{amssymb}
\usepackage{amsthm}
\usepackage[latin1]{inputenc}
\usepackage{eurosym}
\usepackage[dvips]{graphics}
\usepackage{graphicx}
\usepackage{epsfig}

\usepackage{comment,graphicx,esint}

\usepackage[displaymath,mathlines]{lineno}

\allowdisplaybreaks

 \usepackage{xcolor}

\definecolor{labelkey}{rgb}{0,1,0}

\usepackage{ifthen}
\usepackage{enumitem}

\providecommand{\varitem}{} 


\providecommand{\varitem}{} 


\providecommand{\varitem}{} 


\makeindex

\usepackage{fancyhdr}
\usepackage[us,24hr]{datetime}

\usepackage{scrtime}
\usepackage{comment}

\usepackage{tikz}
\usetikzlibrary{arrows,shapes,trees}

\usepackage[english]{babel}

\usepackage{amsfonts}

\usepackage{lmodern}

 \definecolor{Kblue}{rgb}{0,0.651,0.667} 
 
 \definecolor{Korange}{rgb}{0.945,0.561,0}

 \definecolor{Kgreen}{rgb}{0.804,0.808,0}

 \definecolor{Kyellow}{rgb}{0.941,0.71,0}
 
 \definecolor{Mgray}{rgb}{0.5,0.5,0.5}

\definecolor{Mostarda}{rgb}{0.843137,0.823529,0.0627451}

\definecolor{Azul}{rgb}{0.372549,0.447059,0.784314}

\definecolor{Verde}{rgb}{0.0392157,0.568627,0.0784314}

\definecolor{Mblue}{rgb}{0.392157,0.152941,1}

\usepackage{extpfeil}

\usepackage{trimclip}

\DeclareRobustCommand{\rscc}{%
 \clippedrightharpoonup\mathrel{\mspace{-15mu}}\rightharpoonup
}
\makeatletter
\newcommand{\clippedrightharpoonup}{%
  \mathrel{\mathpalette\clipped@rightharpoonup\relax}%
}
\newcommand{\clipped@rightharpoonup}[2]{%
  \sbox\z@{$\m@th#1\mspace{6mu}$}%
  \clipbox{0pt 0pt {\dimexpr\width-.5\wd\z@} 0pt}{$\m@th#1\rightharpoonup$}%
  \smash{\clipbox{{\wd\z@} 0pt 0pt {-\width}}{$\m@th#1\rightharpoonup$}}%
}
\makeatother

\newextarrow{\Rsc}{{-6}{1}{30}{1}}{\relbar\relbar\rscc}

\usepackage[colorlinks=true, pdfstartview=FitV, linkcolor=blue, citecolor=blue, urlcolor=blue,pagebackref=false]{hyperref}

\renewcommand{\div}{\operatorname{div}}

\newcommand{\Cc}{{\mathbb{C}}}
\newcommand{\Nn}{{\mathbb{N}}}

\newcommand{\calH}{{\mathcal{H}}}

\newcommand{\Aa}{{\mathcal{A}}}

\newcommand{\Gg}{{\mathcal{G}}}

\newcommand{\e}{\varepsilon}

\newcommand{\epsi}{\varepsilon}
\def\dd{{\rm d}}
\def\dx{{\rm d}x}
\def\dy{{\rm d}y}

\def\leq{\leqslant}
\def\geq{\geqslant}

\newcommand{\grad}{\nabla}

\newcommand{\ffi}{\varphi} 
\newcommand{\aev}{{a.e.}}

\let\weakly\rightharpoonup

\newcommand{\bR}{{\boldsymbol{R}}}

\newcommand{\RR}{\mathbb{R}}
\newcommand{\NN}{\mathbb{N}}
\newcommand{\ZZ}{\mathbb{Z}}

\newcommand{\su}{\mathcal{U}_\mathcal{A}}
\newcommand{\sv}{\mathcal{V}_\mathcal{A}}
\newcommand{\sw}{\mathcal{W}_\mathcal{A}}

\definecolor{darkgreen}{rgb}{0,0.5,0}
\definecolor{darkblue}{rgb}{0,0,0.7}
\definecolor{darkred}{rgb}{0.9,0.1,0.1}

\newcommand{\sA}{\mathcal{A}}

\newcommand{\rfcomment}[1]{\marginpar{\raggedright\scriptsize{\textcolor{darkblue}{#1}}}}

\newcommand{\R}{{\mathbb{R}}}

\newcommand{\A}{{\mathcal{A}}}

\newcommand{\beq}{\begin{equation}}
\newcommand{\eeq}{\end{equation}}

\numberwithin{equation}{section}

\newtheoremstyle{thmlemcorr}{10pt}{10pt}{\itshape}{}{\bfseries}{.}{10pt}{{\thmname{#1}\thmnumber{
#2}\thmnote{ (#3)}}}
\newtheoremstyle{thmlemcorr*}{10pt}{10pt}{\itshape}{}{\bfseries}{.}\newline{{\thmname{#1}\thmnumber{
#2}\thmnote{ (#3)}}}
\newtheoremstyle{defi}{10pt}{10pt}{\itshape}{}{\bfseries}{.}{10pt}{{\thmname{#1}\thmnumber{
#2}\thmnote{ (#3)}}}
\newtheoremstyle{remexample}{10pt}{10pt}{}{}{\bfseries}{.}{10pt}{{\thmname{#1}\thmnumber{
#2}\thmnote{ (#3)}}}
\newtheoremstyle{ass}{10pt}{10pt}{}{}{\bfseries}{.}{10pt}{{\thmname{#1}\thmnumber{
A#2}\thmnote{ (#3)}}}

\theoremstyle{thmlemcorr}
\newtheorem{theorem}{Theorem}
\numberwithin{theorem}{section}
\newtheorem{lemma}[theorem]{Lemma}
\newtheorem{corollary}[theorem]{Corollary}
\newtheorem{proposition}[theorem]{Proposition}

\theoremstyle{thmlemcorr*}
\newtheorem{theorem*}{Theorem}
\newtheorem{lemma*}[theorem]{Lemma}
\newtheorem{corollary*}[theorem]{Corollary}
\newtheorem{proposition*}[theorem]{Proposition}
\newtheorem{problem*}[theorem]{Problem}
\newtheorem{conjecture*}[theorem]{Conjecture}

\theoremstyle{defi}
\newtheorem{definition}[theorem]{Definition}

\theoremstyle{remexample}
\newtheorem{remark}[theorem]{Remark}

\theoremstyle{ass}

\makeatletter
\def\@seccntformat#1{\@ifundefined{#1@cntformat}%
   {\csname the#1\endcsname\quad}
   {\csname #1@cntformat\endcsname}
}
\makeatother

\allowdisplaybreaks

\usepackage[sort, nocompress, space]{cite}

\address[R.~Ferreira]{King Abdullah University
of Science
and Technology (KAUST), CEMSE Division,  Thuwal 23955-6900,
Saudi Arabia. \bf{E-mail:} rita.ferreira@kaust.edu.sa}

\address[I.~Fonseca]{Department of Mathematical Sciences, Carnegie Mellon University (CMU),
Forbes Avenue,
Pittsburgh, PA 15213, U.S.A.. \bf{E-mail:} fonseca@andrew.cmu.edu}

\address[R.~Venkatraman]{Department of Mathematical Sciences, Carnegie Mellon
University (CMU),
Forbes Avenue,
Pittsburgh, PA 15213, U.S.A.. \bf{E-mail:} rvenkatr@andrew.cmu.edu}

\usepackage[normalem]{ulem}

\DeclareFontFamily{OT1}{pzc}{}
\DeclareFontShape{OT1}{pzc}{m}{it}{<-> s * [0.900] pzcmi7t}{}
\DeclareMathAlphabet{\mathscr}{OT1}{pzc}{m}{it}

\usepackage[scr=boondoxo]{mathalfa}

\usepackage{aurical}
\usepackage[T1]{fontenc}

\usepackage[bbgreekl]{mathbbol}
\usepackage{amsfonts}

\DeclareSymbolFontAlphabet{\mathbb}{AMSb}
\DeclareSymbolFontAlphabet{\mathbbl}{bbold}

\newcommand{\RRn}{\mathbb{R}^\mathbbl{n}}
\newcommand{\RRm}{\mathbb{R}^\mathbbl{m}}
\newcommand{\RRk}{\mathbb{R}^\mathbbl{k}}
\newcommand{\RRd}{\mathbb{R}^\mathbbl{d}}
\newcommand{\RRl}{\mathbb{R}^\mathbbl{l}}

\newcommand{\RRmn}{\mathbb{R}^{\mathbbl{m}\times\mathbbl{n}}}

\newcommand{\RRld}{\mathbb{R}^{\mathbbl{l}\times\mathbbl{d}}}

\newcommand{\bbln}{\mathbbl{n}}
\newcommand{\bblm}{\mathbbl{m}}
\newcommand{\bblk}{\mathbbl{k}}
\newcommand{\bbll}{\mathbbl{l}}
\newcommand{\bbld}{\mathbbl{d}}

\begin{document}

\title[Homogenization of quasicrystals]{Homogenization of Quasi-crystalline functionals \\  via two-scale-cut-and-project
convergence}
\author[\hfill Rita Ferreira, Irene Fonseca, and Raghavendra Venkatraman
\hfill]{Rita Ferreira, Irene Fonseca, and Raghavendra Venkatraman}

\begin{abstract} We consider a homogenization problem associated with quasi-crystalline multiple integrals of the form
\begin{equation*}
\begin{aligned}
u_\epsi\in L^p(\Omega;\RRd) \mapsto  \int_\Omega f_R\Big(x,\frac{x}{\epsi},
u_\epsi(x)\Big)\,\dx, 
\end{aligned}
\end{equation*}
 where \(u_\epsi\) is subject to constant-coefficient linear partial differential constraints. The quasi-crystalline structure of the underlying composite  is encoded in the dependence on the second variable
of  the Lagrangian,   \(f_R\), and is modeled via the {\it cut-and-project} scheme that interprets the heterogeneous microstructure to be homogenized as an irrational subspace of a higher-dimensional space. A key step in our analysis is the characterization of the quasi-crystalline two-scale limits of sequences of the vector fields \(u_\epsi\) that are in the kernel of a given constant-coefficient linear partial differential operator, \(\Aa\), that is, \(\Aa u _\epsi =0\). Our results provide a generalization of  related ones in the literature concerning the \({\rm \Aa =curl } \) case to more
general  differential operators \(\Aa\) with constant coefficients, and  without   coercivity assumptions on the Lagrangian \(f_R\).



\textit{Date:}\ \today

\end{abstract}
%
%
%
%
%
%
%
%

%

\maketitle
\tableofcontents

\section{Introduction}

The  theory of homogenization addresses the description of the macroscopic or effective behavior of a microscopically heterogeneous system. There are multiple applications in the fields of physics, mechanics, materials science and other areas of engineering, including problems aimed at the modeling of composites, stratified or porous media, finely damaged materials, or materials with many holes or cracks.

From the mathematical viewpoint, homogenization is often associated with the study of the asymptotic behavior of oscillating partial differential equations, or of minimization problems deriving from certain oscillating functionals, depending on one or more small-scale parameters that represent the length scales of the
heterogeneities. 

A common assumption in the literature is based on the premise
that the heterogeneities
are evenly distributed, leading to the mathematical assumption of periodicity in the so-called fast variable, which encodes
  the heterogeneities in the mathematical problem.  Even though
the  study of the effective behavior of periodically structured heterogeneous media has enabled the study of  more complex ones, it
is commonly accepted that periodicity is often not the most suited
structural hypothesis. This fact is at the basis of many recent
works devoted to the study of the effective behavior  of random heterogenous materials whose small-length-scale properties are
described at a statistical level only.

Here, we are interested in materials with a quasi-crystalline
microstructure characterized by small-length-scale properties
that are neither periodic nor random.
Quasicrystals, also known as quasiperiodic crystals, are  ordered structures that do not share the translational symmetry of traditional crystals \cite{ShBl85,ShBlGrCa84}. A quasi-crystalline pattern can continuously
fill an \(\bbln\)-dimensional space, but will never be translational symmetric in more than \(\bbln-1\) linearly independent directions.

The discovery of quasicrystals was announced in the early 1980s
   by two groups of crystallographers, Schechtman, Blech, Gratias,
and Cahn \cite{ShBlGrCa84} and Levine and Steinhardt \cite{LeSt84}. At  first,
this was received with scepticism, and even hostility, by the the scientific
community as quasicrystals violate the foundations of classical crystallography. However, in 2011,
Shechtman was awarded the Nobel Prize in Chemistry for this  discovery.
A striking feature of quasicrystals is that their Bragg diffraction
displays peculiar five-, ten-, or twelve-fold symmetry orders
in contrast with the rigid crystallography of periodic crystals.
Moreover, the assembly  of  quasi-crystalline tiling patterns
is  nonlocal and exhibit similar patterns
 at different scales (self-similarity). 

There has been a rich discussion and extensive efforts in various mathematical communities
to   model 
quasicrystals; see \cite{CaTa1,CT2,Lag,Meyer} and related references.
A well-established  mathematical
approach to study quasicrystals is based on aperiodic tilings of hyperplanes, in which one aims at finding a set of geometric shapes, called tiles,  paving the Euclidean plane without gaps or overlaps,   in a non-periodic manner
only (see Figure~\ref{fig:R}).  
A systematic, but not exhaustive, scheme to derive such tilings is via the cut-and-project method, introduce by de Bruijn \cite{Br81}
and further developed by Duneau and Katz \cite{DuKa85}, which extends Penrose's  ideas of  aperiodic tilings of the plane \cite{Pe79} to higher dimensions
(see \cite{BoGuZo10} for a more detailed description).

Roughly speaking, \(\bbln\)-dimensional quasi-crystalline patterns
can be modeled by cutting periodic tilings in an \(\bblm\)-dimensional
space,  with \(\bblm>\bbln\), through an \(\bbln\)-dimensional subspace
with
\textit{irrational slope}. 
To be precise, given
an  \(\bbln\)-dimensional  quaiscrystal \(R\) and representing
by    \(\sigma_R:\RRn\to\RR\) 
a constitutive property of \(R\), we can find \(\bblm\in\NN\),
with  \(\bblm>\bbln\),
 a \(Y^\bblm\)-periodic
function \(\sigma:\RRm\to\RR\) with \(Y^\bblm\subset\RRm\) a parallelotope, and   a
linear map \(\bR:\RRn\to\RRm\) such that
\begin{equation}\label{eq:defqscrystal}
\begin{aligned}
\sigma_R(x) = \sigma(\bR x).
\end{aligned}
\end{equation}
Here, and in the sequel, we do not distinguish the linear map
from its associated matrix  
in $\RR^{\bblm \times \bbln}$, and denote both  by \(\bR\). For
instance, the  matrix
\begin{equation*}
\begin{aligned}
\bR=\frac{1}{\sqrt{2(\tau +\ 2)}}\begin{bmatrix}
1 & \tau & 0 \\
\tau & 0 & 1 \\
0 & 1 & \tau \\
-1 & \tau & 0 \\
\tau & 0 & -1 \\
0 & -1 & \tau \\
\end{bmatrix}, \qquad \text{with }\tau=\frac{1+\sqrt5}{2},
\end{aligned}
\end{equation*}
is associated with the quasi-crystalline phase  \(\rm Al_{63.5}Fe_{12.5}Cu_{24}\)
(see, for instance, \cite{BoGuZo10}).
   
In general, there are multiple choices for \(\bblm\), \(\sigma\),
and \(\bR\), which could lead to some ambiguity in our asymptotic
analysis. However, as proved in  \cite{BoGuZo10}, the homogenization
analysis does not depend on \(\bR\) provided it satisfies the
following diophantine condition
 \begin{equation}
\label{eq:Rcriterion}
\begin{aligned}
\bR^*k\not=0\enspace \text{for all } k \in \ZZ^\bblm\backslash\{0\},
\end{aligned}
\end{equation}
where 
 $\bR^* $ denotes the transpose of $\bR$. This condition implies
that some entries of \(\bR\) must be irrational, which justifies
the  expression
\textit{irrational slope} used above.

\begin{figure}[h]
  \centering
    \includegraphics[width=0.3\textwidth]{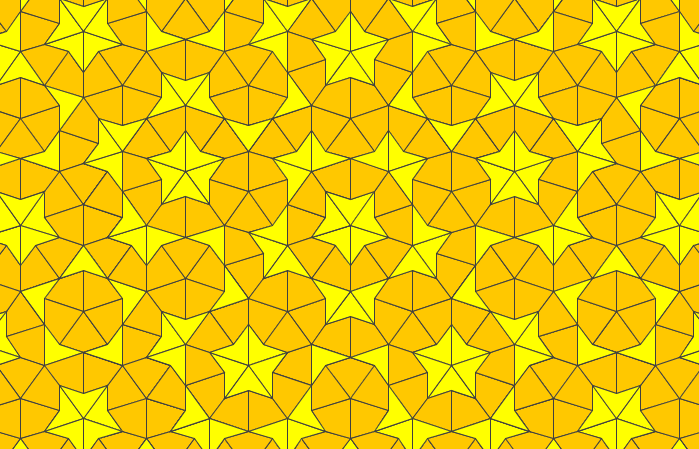}
    \caption{A quasi-crystalline heterogeneous microstructure corresponding to the so-called Penrose tiling of the plane with five-fold symmetry. Image source: wikipedia.}\label{fig:R}
\end{figure}

%

Quasi-crystalline composites and alloys have played a central role in materials science and other areas of engineering \cite{BlBaOtSh00,
KeBoDuFo12, GuSuPa20}. Indeed, Al-Cu-Fe quasi-crystalline materials in polymer-based composites have significantly shown to improve wear-resistance to volume loss, and a two-fold increase in the elastic moduli. As we mentioned before, the mathematical study of such quasi-crystalline composites does not fit within the \textit{classical} periodic
homogenization theory. More appropriate in the context of quasicrystal composites are almost-periodic and stochastic homogenization,
which were initiated with the works of Papanicolaou and Varadhan \cite{PaVa81} and Kozlov \cite{Ko79}, for partial differential equations, and Dal Maso and Modica \cite{DaMo860}
within a variational framework; see also \cite{Br92,De92,CaGa02}
and the references therein.
However, such approaches often lead to untractable formulas that
do not take full advantage of the quasi-crystalline feature of
the problem. Instead, we adopt and further develop a homogenization procedure based on the  
two-scale-cut-and-project
convergence introduced in \cite{BoGuZo10}, and recently revisited in
\cite{ WeGuCh19}.
\medskip

In this paper, we  initiate
a research program devoted to the study
of quasi-crystalline homogenization problems involving oscillating
integral energies under quasi-crystalline oscillating differential constraints, in the framework of \(\Aa\)-quasiconvexity. To be precise, we aim at characterizing the asymptotic behavior of integral energies of the form
\begin{equation}\label{eq:qscrystFe}
\begin{aligned}
F_\epsi (u_\epsi):= \int_\Omega f_R\Big(x,\frac{x}{\epsi^\alpha},
u_\epsi(x)\Big)\,\dx
\end{aligned}
\end{equation}
as \(\epsi\to 0^+\), where \(\epsi^\alpha>0\), with \(\alpha\geq0\), represents the length-scale
of the tiles featuring the quasi-crystalline composite. Moreover, \(\Omega\subset\RRn\)
with, \(\bbln\in\NN\), is an  open and bounded set that represents the container occupied
by the composite, and \(f_R\) is the Lagrangian of the system
whose dependence in the second variable, the fast variable,
encodes the quasi-crystalline structure of the composite, highlighted
with the subscript \(R\) as in \eqref{eq:defqscrystal}. Finally, \(u_\epsi\) is an abstract vector-valued order-parameter whose physical interpretation might depend on the problem in question. A typical case is that  in which \(u_\epsi\) is curl-free, \(u_\epsi = \grad v_\epsi\) for some potential deformation \(v_\epsi\).  However, many applications require that \(u_\epsi\) instead satisfies other linear partial differential constraints, such as Maxwell's equations in the case of electromagnetism,  or, in the case of linear elasticity, \(u_\epsi\) is the symmetric part of a gradient. A unified abstract approach to several of these constraints is that of \(\Aa\)-free fields, as pioneered by Fonseca and M\"uller \cite{FoMu99}  (see also  \cite{Dac1,Dac2,Sa04}). To be precise, \(u_\epsi \in L^p(\Omega;\RRd)\) 
is subject to quasi-crystalline oscillating differential constraints such as 
\begin{equation*}
\begin{aligned}
\Aa_\epsi u_\epsi:=\sum_{i=1}^\bbln A^i_R\Big(\frac{\cdot}{\epsi^\beta}\Big) \frac{\partial
u_\epsi}{\partial x_i}(\cdot) \to 0\enspace \text{ strongly in
 \(W^{-1,p}(\Omega;\RRl)\) }
\end{aligned}
\end{equation*}
or, in divergence form, 
\begin{equation*}
\begin{aligned}
\Aa_\epsi u_\epsi:=\sum_{i=1}^\bbln \frac{\partial
}{\partial x_i}\Big( A^i_R\Big(\frac{\cdot}{\epsi^\beta}\Big) 
u_\epsi(\cdot)\Big) \to 0\enspace \text{ strongly in
 \(W^{-1,p}(\Omega;\RRl)\)}
\end{aligned}
\end{equation*}
with \(\bbld,\,\bbll\in\NN\) and \(1<p<\infty\), where for every \(x\in \RRn\), \(A^i_R(x)\in \text{Lin}(\RRd;\RRl)\)
features a quasi-crystalline pattern, and \(\beta\geq0\) is a parameter. For the study of homogenization of integral energies with periodic energy densities and under periodically oscillating  \(\Aa\)-free differential constraints, we refer the reader  to \cite{MaMoSa15,KrKr16,DaFo16,DaFo160,BrFoLe00,FoKr10}.
 
 As in the periodic setting \cite{DaFo16,DaFo160}, we expect different asymptotic regimes according the ratio between \(\alpha\) and \(\beta\). As a starting
point to this extensive research project, we first focus here on the
case where \(\beta=0\) and \(A^i_R\) is independent of \(x\), in which
case \(u_\epsi \) is subjected to homogeneous first-order linear partial differential constraints.  Precisely, in this manuscript
we address the problem of characterizing the asymptotic behavior
as \(\epsi\to 0^+\) of  
integral energies of the form
\begin{equation}\label{eq:defFe}
\begin{aligned}
F_\epsi (u):= \int_\Omega f_R\Big(x,\frac{x}{\epsi},
u(x)\Big)\,\dx
\end{aligned}
\end{equation}
 for \(u \in L^p(\Omega;\RRd)\) satisfying \(\Aa u =0\), where
\begin{equation}\label{eq:Aconstraint}
\begin{aligned}
\Aa u:= \sum_{i=1}^\bbln A^{(i)}\frac{\partial u}{\partial x_i}\quad
\text{with } A^{(i)}\in \RRld \text{ for all } i\in\{1,...,\bbln\}.
\end{aligned}
\end{equation}

We refer to Section~\ref{subs:Aoperators} for a rigorous definition
of the identity \(\Aa u =0\),
in which case we say that the vector field \(u\) is $\Aa$-free
(see Definition~\ref{def:Afree}). A common
assumption
within studies involving \(\Aa\)-free vector fields
is the constant-rank property, which states that  there exists \(\mathscr{r}\in\NN\) such that for all \(w\in\RRn\setminus
\{0\}\), we have
\begin{equation}
\label{constrank}
\begin{aligned}
{\rm rank}\,\mathbb{A}(w) = \mathscr{r},
\end{aligned}
\end{equation}
where \(\mathbb{A}:\RRn \to \RRld\) denotes the symbol of \(\Aa\),
and is defined by 
\begin{equation}\label{eq:Asymb}
\begin{aligned}
\mathbb{A}(w):= \sum_{i=1}^\bbln A^{(i)}w_i
\end{aligned}
\end{equation}
for  \(w\in\RRn\). We assume that  our operator \(\A\) satisfies
the constant-rank property, and we refer
the reader to \cite{FoMu99,Ta79,Mu81} for further insights on
this property and on \(\Aa\)-free
fields.

Our asymptotic analysis of the energy integrals in \eqref{eq:defFe}
under the constraint \eqref{eq:Aconstraint} is based on  \(\Gamma\)-convergence
techniques, whose key point is to find  an integral representation to
\begin{equation}\label{eq:defFhom}
\begin{aligned}
F_{\rm hom}(u):=\inf \Big\{ \liminf_{\epsi\to0^+}  F_\epsi(u_\epsi)\!:\, u_\epsi \weakly  u \text{ in } L^p(\Omega;\RRd),
\enspace \Aa u_\epsi =0\Big\}.
\end{aligned}
\end{equation}
To state our main theorem regarding this integral representation, we first introduce the hypotheses on the Lagrangian,  $f_R: \Omega \times \RRn \times \RRd
\to [0,\infty)$: 
\begin{enumerate}
    \item[(H1)] (Quasi-crystallinity:) there exist \(\bblm\in
    \NN\), with \(\bblm>\bbln\), a matrix \(\bR \in \RRmn\) satisfying
\eqref{eq:Rcriterion}, and a continuous       function
$f :\Omega \times \RRm \times \RRd \to [0,\infty)$ such that
    the function $f(x, \cdot, \xi)$ is $Y^\bblm$-periodic for
each $(x,\xi) \in \Omega \times \RRd,$ with $Y^\bblm$ denoting
a paralleletope in $\RRm, $ and\[f_R(x,z,\xi) = f(x,\bR z , \xi)\] for all $(x,z,\xi) \in \Omega \times \RRn \times \RRd.
$  

    \item[(H2)] (Growth:) there exist $p \in (1,\infty)$  and
\(C>0\) such
that 
    \begin{align*}
        0 \leqslant f_R(x,z,\xi) \leqslant C(1 + |\xi|^p) 
    \end{align*}
 for all  $(x,z,\xi) \in \Omega \times \RRn \times \RRd$.
\end{enumerate}
In the proof of the lower bound for the integral representation
of \(F_{\rm hom}\), we will require, in addition,

\begin{enumerate}
      \item[(H3)] (Convexity:) for all   $(x,y) \in \Omega\times \RRm$, the function $\xi \mapsto f(x, y, \xi)$
is convex and $C^1$. 
\end{enumerate}  

We refer the reader to Section~\ref{sect:notpre} for a  list of the main notations we use in this manuscript. However, for the readability of our main results, we  clarify upfront that   $ L^p_\#(Y^\bblm; \RRn)$ denotes
the space of  $Y^\bblm$-periodic functions  belonging to   \(L^p_{\rm loc}(\RRm)\). Moreover, given a Lebesgue measurable  set $B\subset \RRk $, with \(\bblk\in\NN\),
we use the notation $\fint_B \cdot$ in place of  $ \frac{1}{\mathcal{L}^\bblk(B)}\int_B \cdot$,
where $\mathcal{L}^\bblk(B)$ denotes
the \(\bblk\)-dimensional Lebesgue measure of  $B$.

\begin{theorem}\label{thm:main} Let \(\Omega\subset\RRn\) be
an open and bounded set,  let 
 $f_R: \Omega \times \RRn \times \RRd
\to [0,\infty)$ be a function satisfying (H1)--(H3), let \(F_{\rm
hom}\) be the functional introduced in \eqref{eq:defFhom}, and  assume
that \eqref{constrank} holds. Then, for
all
\begin{align}\label{eq:setUA}
  u\in  \su := \big\{u \in L^p(\Omega;\RRd) : \Aa u = 0\big \},
\end{align}
we have
\begin{equation*}
\begin{aligned}
F_{\rm hom} (u)=\int_\Omega f_{\rm hom} (x,u(x))\,\dx,
\end{aligned}
\end{equation*}
where
\begin{equation*}
\begin{aligned}
f_{\rm hom} (x,\xi):=\inf_{v\in\sv } \fint_{Y^\bblm} f(x,y,\xi + v(y))\,\dy
\end{aligned}
\end{equation*}
with
\begin{equation}
\begin{aligned}\label{eq:setVA}
    \sv := \bigg\{v \in L^p_\# (Y^\bblm; \RRd)\! :\, v \mbox{
is }  \Aa_{\bR^*}\mbox{-free in the sense of Definition~\ref{def:AR*free} and }\int_{Y^\bblm} v(y)\,\dy=0\bigg\}.
\end{aligned}
\end{equation}

\end{theorem}

\begin{remark}[On the hypotheses of Theorem~\ref{thm:main}]\label{rmk:onH}
(i)~In the homogenization literature, measurability of \(f\) with respect to the fast-variable
is often preferred over continuity. As we further discuss in Section~\ref{subs:Rmatrix}, 
measurability of \(f_R\)  requires, in general, Borel-measurability
of \(f\). A common approach to deal with lack
of continuity is to combine periodicity with  Scorza--Dragoni's type results that,
up to a set of small measure, allow to reduce the problem
to the continuity setting. Here, however, we cannot use such an
argument because a  set of small \(\bblm\)-dimensional Lebesgue measure, the ambient space for the fast variable in terms of
(the periodic function) \(f\), may not have small \(\bbln\)-dimensional Lebesgue, the ambient space for the fast variable in terms of
(the quasi-crystalline function) \(f_R\).
(ii)~The non-convex case raises non-trivial difficulties in the
quasi-crystalline setting, and will be the subject of a forthcoming
work.\end{remark}

 In the Sobolev setting, homogenization of integral energies of the form \eqref{eq:defFe}
 under {\em non-periodic} assumptions was undertaken in \cite{Br92,DaMo860,
 HeNeVa18}
 in the \(\Aa=\rm curl\) case,  assuming coercivity. Within
 the quasi-crystalline framework, Theorem~\ref{thm:main}
 extends these results to the general \(\Aa\)-free setting and
without coercivity.
We prove Theorem~\ref{thm:main} in Section~\ref{sect:proofmain}; 
the main tools we use here  are based on \(\Gamma\)-convergence
 and on two-scale convergence
adapted to the quasi-crystalline setting, also called   
two-scale-cut-and-project
convergence. For brevity, and having in mind the relation \eqref{eq:defqscrystal}, we refer to the   
two-scale-cut-and-project
convergence as \(\bR\)-two-scale convergence. This notion was introduced in  \cite{BoGuZo10} (also see \cite{WeGuCh19})
as an extension of the usual notion of two-scale convergence
\cite{Al92,Ng89}
to enable the study of composites whose underlying microstructure
has a quasi-crystalline feature. 

Here, we further extend the
study of  \(\bR\)-two-scale convergence in two different ways.
In \cite{BoGuZo10, WeGuCh19}, the authors consider sequences
in \(L^2\) and their arguments are based on Fourier analysis
relying heavily on Parseval's
and Plancherel's identities. Instead, we consider the more general
case of \(L^p\) with \(p\in (1,\infty)\). Moreover,
in \cite{BoGuZo10} the authors characterize the limit, with
respect
to the \(\bR\)-two-scale convergence, of bounded sequences in
\(W^{1,2}\), while in \cite{WeGuCh19} the authors  characterize
the limit associated with bounded sequences in \(L^2\) that are
divergence-free or curl-free. Here, besides generalizing these
results to the \(L^p\) case, we provide a unified approach to
all these cases by considering bounded sequences in \(L^p\) that
are \(\Aa\)-free, in the spirit of \cite{FoKr10} concerning the
periodic case.

Next, we state our main result regarding the characterization
of the limits of bounded sequences in \(L^p\) that
are \(\Aa\)-free. We refer the reader to Sections~\ref{subs:Aoperators} and
\ref{subs:2scAfree}, where
we give a precise meaning to the expressions ``\(\Aa\)-free''
and ``\((\Aa,\Aa_{\bR^*}^y)\)-free'' that we make use in this statement.

\begin{theorem}\label{thm:main2}
Let \(\bR\in \RRmn \) satisfy \eqref{eq:Rcriterion}. A
function
\(u\in L^p(\Omega\times Y^\bblm;\RRd)\) is the \(\bR\)-two-scale
limit of an \(\Aa\)-free sequence \(\{u_\epsi\}_\epsi\subset L^p(\Omega;\RRd)\) if and only if   \(u\) is \((\Aa,\Aa_{\bR^*}^y)\)-free in the sense of
Definition~\ref{def:AAR*free}; that is, 
\begin{equation}\label{eq:charR2sclim}
\begin{aligned}
\Aa \bar u_0 =0\quad \text{and} \quad \Aa_{\bR^*}^y \bar u_1
=0
\end{aligned}
\end{equation}
in the sense of Definition~\ref{def:Afree} and Definition~\ref{def:AR*free},
respectively, where  \(\bar
 u_0:=\int_{Y^\bblm} u(\cdot,y)\,\dy\) and \(\bar u_1:= u - \bar
u_0\). 
\end{theorem}

We prove Theorem~\ref{thm:main2} in Section~\ref{subs:2scAfree},
where we use similar arguments to those in  \cite{FoKr10} concerning the
periodic case.
We observe that the sufficient part in Theorem~\ref{thm:main2},
which guarantees that \eqref{eq:charR2sclim} fully characterizes the \(\bR\)-two-scale
limits, is new in the literature even for \(p=2\) and \(\Aa=\rm
curl\) or \(\Aa=\rm div\) treated in  \cite{BoGuZo10, WeGuCh19}.
Furthermore, in Section~\ref{sec:counting} we give an alternative proof of Theorem~\ref{thm:main2} for the \(\Aa=\rm curl\) case using  arguments
 based on Fourier analysis
that differ from those in \cite{BoGuZo10, WeGuCh19} because Parseval's
and Plancherel's identities do not hold for \(p\not=2\).
This
alternative proof provides an equivalent alterative characterization
for the \(\bR\)-two-scale limit of bounded sequences in \(W^{1,p}\),
and may provide  useful
arguments to study homogenization problems involving  quasi-crystalline functionals  in the \(\Aa=\rm curl\) case. {\color{blue}{}} This alternative characterization can be stated as follows. 

 \begin{theorem}\label{thm:alterchaRl}
 Let \(\bR\in \RRmn \) satisfy \eqref{eq:Rcriterion} and let \(Y^\bblm\subset\RRm\)
be a parallelotope. Then, a
function
\(v\in L^p(\Omega\times Y^\bblm;\RRn)\) is the \(\bR\)-two-scale
limit of a sequence \(\{\grad v_\epsi\}_\epsi\) with    \(\{v_\epsi\}_\epsi \)
bounded in \(W^{1,p}(\Omega)\) if and only if there exist \(v_0\in W^{1,p}(\Omega) \) and
\(v_1\in L^p(\Omega; \Gg_{\bR}^p)\)
such that
\begin{equation*}
\begin{aligned}
v= \grad v_0 + v_1,
\end{aligned}
\end{equation*}%
where 
\begin{equation}
\label{eq:spaceGR}
\begin{aligned}
\Gg_{\bR}^p:= \Big\{ w \in L^p_\#(Y^\bblm; \RRn)\!: \,   \hat
w_k = \lambda_k \bR^*k  \text{ for some } \{\lambda_k\}_{k\in
\ZZ^\bblm } \subset \Cc \text{ with } \lambda_0=0\Big\}
\end{aligned} 
\end{equation}
with  \(\hat w_k:= \fint_{Y^\bblm} w(y)e^{-2\pi i k \cdot
y} \, \dy\), \(k\in \ZZ^\bblm \), denoting the Fourier coefficients
of \(w\). 
\end{theorem}

\begin{remark}\label{rmk:reldifchar} We recall that if \(u_\epsi\in L^p(\Omega;\RRn)\) is  \(\rm curl\)-free in \(\RRn\) with \(\Omega\) simply connected, then there exists \(v_\epsi\in
W^{1,p}(\Omega)\) such that \(u_\epsi = \grad v_\epsi\). Thus,
in terms of the notations in the two previous results with \(\bbld=\bbln\), we have
\(\bar u_0 = \grad v_0\) and \(\bar u_1 = v_1\). In particular,
\eqref{eq:spaceGR} provides an alternative characterization of
\(\Aa_{\bR^*}\)- and \(\Aa_{\bR^*}^y\)-free vector fields introduced in Definition~\ref{def:AR*free} in the \(\Aa=\rm curl\) case
(also see Remark~\ref{rmk:relGRAR*} for a detailed argumentation). \end{remark}

\section{Notation and Preliminaries}\label{sect:notpre}

Throughout this manuscript,  \(\bblm,\,\bbln\in\Nn\) are
such that \(\bblm> \bbln\),  
\(\Omega\subset
\RRn \) is an open and bounded set,  \(Y^\bblm\) is a parallelotope
in \(\RRm\), \(\Pi\subset\RRn\) is a parallelotope
in \(\RRn\), \(\bbld\), \(\bbll\in\NN\), and   \(p\), \(p'\in (1,\infty)\) are such that \(\frac1p
+\frac1{p'}=1\). Moreover, we assume that
\(\epsi\) takes values on an arbitrary sequence of positive numbers
that converges to zero.

We use the subscript \(\#\) within function spaces to highlight
an underlying periodicity, in which case the domain indicates the periodicity cell. For instance, \(C_\#(Y^\bblm)=
\{u\in C(\RRm)\!:\, u \text{ is \(Y^\bblm\)-periodic}\} \) and
\(L^p_\#(\Pi)=
\{u\in L^p_{\rm loc}(\RRn)\!:\, u \text{ is \(\Pi\)-periodic}\} \).
Moreover, given a Lebesgue measurable  set $B\subset \RRk $, with \(\bblk\in\NN\),
we use the notation $\fint_B \cdot$ in place of  $ \frac{1}{\mathcal{L}^\bblk(B)}\int_B
\cdot$,
where $\mathcal{L}^\bblk(B)$ denotes
the \(\bblk\)-dimensional Lebesgue measure of  $B$. 

Next, we  compile the notation and main properties of the cut-and-project maps {\bf\textit{R}} and differential operators \(\Aa\)
introduced in the Introduction, and that we make use in the sequel.

\subsection{Cut-and-project maps {\bf\textit{R}}}\label{subs:Rmatrix}

 In this paper,   \(\bR:\RRn\to \RRm\) is a linear map, whose
associated matrix in \(\RRmn\) is also denoted by \(\bR\). We do not distinguish
between the transpose matrix and the adjoint  of \(\bR\), and denote both
 by $\bR^*$.  We often assume that the   criterion
\eqref{eq:Rcriterion} on \(\bR\),
\begin{equation*}
\begin{aligned}
\bR^*k\not=0\enspace \text{for all } k \in \ZZ^\bblm\backslash\{0\},
\end{aligned}
\end{equation*}
holds, in which case we   refer to it explicitly.

As shown in \cite{BoGuZo10}, if \(g:\RRm\to\RR\) is a trigonometric polynomial, then the ergodic mean of \(g\circ \bR\),  \(\mathscr{M}(g\circ \bR)\), is uniquely defined provided that \(\bR\) satisfies
\eqref{eq:Rcriterion}, in which case we have
\begin{equation*}
\begin{aligned}
\mathscr{M}(g\circ \bR):= \lim_{\tau \to \infty} \frac{1}{2\tau^\bbln} \int_{(-\tau,\tau)^\bbln} g(\bR x)\,\dx = \fint_{Y^\bblm} g(y)\,\dy,
\end{aligned}
\end{equation*}
where \(Y^\bblm\) is a parallelotope in \(\RRm\)
representing the periodicity cell of \(g\).

Throughout this manuscript, we consider functions  \(\sigma_R\) as   in  \eqref{eq:defqscrystal}. We 
 observe 
that such definition raises measurability issues. In
fact, we can only guarantee that  \(\sigma_R\) in \eqref{eq:defqscrystal}\ is measurable
provided that \(\sigma\) is Borel-measurable. We conjecture that there are functions  \(\sigma\in L^\infty(\RRm)\) for which the corresponding 
function \(\sigma_R\) in \eqref{eq:defqscrystal}\ is not  measurable.
This conjecture is based upon the observation that the pre-image of a measurable set \(B\subset\RRm\) through \(R\), \(R^{-1}(B)\), acts as a projection of the set \(B\) onto the lower-dimensional space \(\RRn\); moreover,
as it is well-know, the projection of a measurable set may not be measurable. To overcome this issue,  we take in 
\eqref{eq:defqscrystal}  the Borel representative of \(\sigma\).

\subsection{Differential operators \(\Aa\) with constant coefficients}\label{subs:Aoperators}
We consider homogeneous first-order  linear partial differential
operators with  constant coefficients, \(\Aa\), 
that map  $u : \Omega \to \RRd $ into $\sA u:  \Omega \to \RRl, $ of the form
\begin{equation*}
\begin{aligned}
\Aa u:= \sum_{i=1}^\bbln A^{(i)}\frac{\partial u}{\partial x_i}\quad
\text{with } A^{(i)}\in \RRld \text{ for all } i\in\{1,...,\bbln\}.
\end{aligned}
\end{equation*}
The formal  adjoint of \(\Aa\), which we denote by $\sA^*$,  maps $v : \Omega
\to \RRl$ into $\sA^* v : \Omega \to \RRd$ and is defined by
\begin{align*}
    \sA^* v:= - \sum_{i=1}^\bbln \big(A^{(i)}\big)^T
\frac{\partial v}{\partial x_i}. 
\end{align*}

We observe that \(\Aa\) can be viewed as a
bounded, linear operator
\(\Aa: L^p(\Omega;\RRd) \to W^{-1,p}(\Omega;\RRl)\) by setting
\begin{equation*}
\begin{aligned}
\langle \Aa u, v \rangle := \int_\Omega u \cdot \Aa^*v\,\dx
\end{aligned}
\end{equation*}
for all \(u\in L^p(\Omega;\RRd)  \) and \(v\in W^{1,p'}_0(\Omega;\RRl) \).
We observe further that if $u \in C^1_c(\Omega;\RRd)$ and  $v \in C^1_c(\Omega; \RRl)$, then
\begin{align*}
    \int_\Omega \sA u \cdot v \,\dx = \int_\Omega u \cdot \sA^*
v \,\dx 
\end{align*}
by integration by parts. Similarly, if $u \in C^1_\#(\Pi;\RRd)$ and  $v
\in C^1_\#(\Pi;\RRl)$, 
then
\begin{align*}
    \int_\Pi \sA u \cdot v \,\dx = \int_\Pi u \cdot \sA^*
v \,\dx. 
\end{align*}

We assume  that \(\A\) satisfies the constant-rank property, that
is, there exists \(\mathscr{r}\in\NN\) such that for all \(w\in\RRn\setminus
\{0\}\), we have \({\rm rank}\,\mathbb{A}(w) = \mathscr{r},\)
where \(\mathbb{A}:\RRn \to \RRld\) denotes the symbol of \(\Aa\), and is defined by \eqref{eq:Asymb}.
 As we mentioned in the Introduction, the constant-rank property is a  common assumption
within studies involving \(\Aa\)-free vector fields. We refer
the reader to \cite{FoMu99,Ta79,Mu81} for further insights on this property and on \(\Aa\)-free
fields, whose notion we recall next. 

\begin{definition}[$\sA$-free fields]\label{def:Afree}
(i) Given  $u \in L^p(\Omega; \RRd)$,  
we say that $\sA u$ \textit{exists
in }$L^p(\Omega;\RRl)$ if there exists a function $U \in L^p(\Omega;\RRl)$
such that, for every $\phi \in C^1_c(\Omega;\RRl)$, we have 
\begin{align}\label{Au=0Omega}
    \int_\Omega u \cdot \sA^* \phi \,\dx = \int_\Omega U \cdot
\phi \,\dx. 
\end{align}
In this case, we write $\sA u := U.$ We say that $u$ is $\sA$-free,
and write $\sA u = 0$, if  \eqref{Au=0Omega} holds with $U =0$.

(ii) Given  $v \in L^p_\#(\Pi; \RRd)$, 
we say that $\sA v$ \textit{exists
in }$L^p_\#(\Pi;\RRl)$ if there exists a function $V \in L^p_\#(\Pi;\RRl)$
such that, for every $\varphi \in C^1_\#(\Pi;\RRl)$, we have 
\begin{align}\label{Au=0Pi}
    \int_\Pi u \cdot \sA^* \varphi \,\dy = \int_\Pi V \cdot
\varphi \,\dy. 
\end{align}
In this case, we write $\sA v := V.$ We say that $v$ is $\sA$-free,
and write $\sA v = 0$, if  \eqref{Au=0Pi} holds with $V =0$.

\end{definition}

\begin{remark}[$\sA$ applied to vector fields depending on several variables]\label{rmk:Aseveralv}
Whenever a vector field  depends on two or more variables, we
index \(\Aa\) with the underlying variable to which \(\Aa\) is
being applied to the vector field. For instance, if \(u=u(x,y)\), then
\(\Aa_x u\) refers to \(\Aa\) applied to \(u\) as a function
of \(x\) with \(y\) regarded as a fixed parameter. Similarly,
\(\Aa_y u\) refers to \(\Aa\) applied to \(u\) as a function
of \(y\) with \(x\) regarded as a fixed parameter.\end{remark}

A crucial
result in the variational theory of $\sA$-free fields is the
following $\sA$-free periodic extension lemma, established in
 \cite[Lemma~2.15]{FoMu99}.
We make repeated use of a similar statement, also proved in 
\cite[Lemma~2.8]{FoKr10},  and hence record it here for the
readers' convenience. 

\begin{lemma}[$\sA$-free periodic extension]\label{lem:2.8FK}
Let $\Pi \subset \RRn$
be a parallelotope, let  $O \subset \Pi$ be an open set, let $1 < p < \infty$, and assume that $\sA$
satisfies \eqref{constrank}. Let $\{v_n\} \subset L^p(O;\RRd)$ be a p-equiintegrable sequence in \(O\), with    $v_n \rightharpoonup
0$ in $L^p(O;\RRd)$ and $\sA v_n \to 0$ in $W^{-1,p}(O;\RRl)$.
Then, 
there exist  an $\sA$-free sequence $\{u_n\} \subset L^p_{\#}(\Pi;\RRd)$,
that is $p$-equiintegrable in $\Pi$, and  a positive constat \(C=C(\sA)\)   such that 
\begin{align*}
    &u_n - v_n \to 0 \mbox{ in } L^p(O;\RRd),\quad u_n \to 0 \mbox{
in } L^p(\Pi \backslash O; \RRd),\quad \fint_{\Pi} u_n \,\dy  = 0,\\
    &\|u_n\|_{L^p(\Pi;\RRd)} \leqslant C \|v_n\|_{L^p(O;\RRd)} \mbox{
for all } n \in \mathbb{N}. 
\end{align*}
\end{lemma}
\begin{proof}
The proof of this lemma with $\Pi = (0,1)^\bbln$ can be found in \cite[Lemma~2.15]{FoMu99}
and \cite[Lemma~2.8]{FoKr10}. 
 The case in which \(\Pi \) is an arbitrary parallelotope
 follows
by an affine change of variables. 
\end{proof}

\section{Cut-and-project-two-scale convergence}\label{sect:cap2sc}

The notion of two-scale convergence was first introduced in the
\(L^2\) setting by Nguetseng \cite{Ng89}, and further developed by Allaire \cite{Al92}. Initially, it was used to provide a mathematical rigorous justification of the formal asymptotic expansions that are commonly adopted in the study of homogenization problems.
Posteriorly, the notion of two-scale convergence was extended, in particular,
to \(L^p\),   \(L^1\),  $BV$, and Besicovitch spaces
  \cite{LuNgWa02, BuFo15, Am98,  FeFo12, CaGa02},  and also to the multiple-scales case \cite{AlBr96,FoZa03,FeFo120},
that enhanced several variational homogenization studies hinged
on a \(\Gamma\)-convergence approach, such
as \cite{BaFo07,CiDaDe04,FoZa03 , FeFo12, Ne12}.

In this section,  we first
address the study of the notion of two-scale convergence
in the quasi-crystalline setting, which we refer to as cut-and-project-two-scale
convergence (or, for brevity,   \(\bR\)-two-scale
convergence), with \(\bR\) as in Section~\ref{subs:Rmatrix}. We then prove Theorem~\ref{thm:main2}. 

As we mentioned in the Introduction, the  \(\bR\)-two-scale convergence was introduced in  \cite{BoGuZo10} (also see \cite{WeGuCh19})
as an extension of the usual notion of two-scale convergence
to enable the study of composites whose underlying microstructure
has a quasi-crystalline feature. 
Using arguments  based on Fourier analysis,  the authors in  \cite{BoGuZo10} characterize the limit, with respect
to the \(\bR\)-two-scale convergence, of bounded sequences in
\(W^{1,2}\), while the authors in \cite{WeGuCh19}   characterize the limit associated with bounded sequences in \(L^2\) that are divergence-free or curl-free. Here, besides generalizing these results to the \(L^p\) setting, with \(1<p<\infty\), we provide a unified approach to all these cases by considering bounded sequences in \(L^p\) that are \(\Aa\)-free, with \(\A\) as in Section~\ref{subs:Aoperators}. Our arguments are close to those in   \cite{FoKr10} concerning
the
periodic case, and are hinged on properties of \(\Aa\)-free vector
fields.

We first introduce the definition of \(\bR\)-two-scale convergence
in \(L^p(\Omega;\RRk)\). We make use of the results in this section
with  \(\bblk\) equal to
either \(1\), \(\bbld\),  \(\bbll\), or \(\bbln\).
\begin{definition}[\(\bR\)-two-scale convergence]\label{def:R2sc}
We say that a sequence \(\{u_\epsi\}_{\epsi} \subset L^p(\Omega;\RRk)\) \(\bR\)-two-scale converges to a function \(u\in L^p(\Omega\times Y^\bblm;\RRk)\) if for all \(\ffi \in L^{p'}(\Omega;C_\#(Y^\bblm;\RRk))\) we have
\begin{equation}\label{eq:lim2sc}
\begin{aligned}
\lim_{\epsi\to0^+}\int_\Omega u_\epsi(x)\cdot  \ffi\Big(x, \frac{\bR x}{\epsi} \Big)\,\dx = \int_\Omega \fint_{Y^\bblm} u(x,y)\cdot  \ffi(x,y)\,\dx\dy
\end{aligned}
\end{equation}
and we write \(u_\epsi\Rsc{\bR\text{-}2sc} u \). 
\end{definition}

\begin{remark}[Uniqueness of  \(\bR\)-two-scale limits]\label{rmk:uniq}
There is uniqueness of the \(\bR\)-two-scale limit. In fact, if  \(\{u_\epsi\}_{\epsi} \subset L^p(\Omega;\RRk)\) and  \(u\), \(
 \tilde u\in L^p(\Omega\times
Y^\bblm;\RRk)\)  are such that  \(u_\epsi\Rsc{\bR\text{-}2sc} u \) and
 \(u_\epsi\Rsc{\bR\text{-}2sc} \tilde u \), then
\begin{equation*}
\begin{aligned}
\int_\Omega \fint_{Y^\bblm} (u(x,y) - \tilde u(x,y))\cdot  \ffi(x,y)\,\dx\dy
\end{aligned}
\end{equation*}
for all  \(\ffi \in L^{p'}(\Omega;C_\#(Y^\bblm;\RRk))\). Hence, \(u=\tilde
u\) \aev\ in \( \Omega\times Y^\bblm\).
\end{remark}

\begin{remark}[On the test functions for \(\bR\)-two-scale convergence]\label{rmk:testfct}
Assume that  \(\{u_\epsi\}_{\epsi}\) is a bounded sequence in
\(L^p(\Omega;\RRk)\). Then,   \(\{u_\epsi\}_{\epsi} \)  
\(\bR\)-two-scale
converges to a function \(u\in L^p(\Omega\times Y^\bblm;\RRk)\) if and only if  \eqref{eq:lim2sc} holds for all \(\ffi \in C_c^\infty(\Omega; C_\#^\infty(Y^\bblm;\RRk))\). To prove this statement, it suffices to use the density of  \(C_c^\infty(\Omega; C_\#^\infty(Y^\bblm;\RRk))\)  in \(L^{p'}(\Omega;C_\#(Y^\bblm;\RRk))\) and the boundedness of \(\{u_\epsi\}_\epsi\) in \(L^p(\Omega;\RRk)\).
\end{remark}

The next two propositions characterize the relationship between
the \(\bR\)-two-scale limit and the usual weak and strong limits
in \(L^p(\Omega;\RRk)\).  
\begin{proposition}\label{prop:R2sc-wc} Assume that  \(\{u_\epsi\}_{\epsi}
\subset L^p(\Omega;\RRk)\) is a sequence that \(\bR\)-two-scale
converges to a function \(u\in L^p(\Omega\times Y^\bblm;\RRk)\).
Then,  \(u_\epsi
\weakly \bar u_0\) weakly in \(L^p(\Omega;\RRk)\), where \(\bar
u
_0(\cdot)
:= \fint_{Y^\bblm} u(\cdot,y)\,\dy\). In particular,   \(\{u_\epsi\}_{\epsi}\)
is bounded in \(L^p(\Omega;\RRk)\).
\end{proposition}

\begin{proof}
Let   \(\phi \in L^{p'}(\Omega;\RRk)\), and set  \(\ffi(x,y):=
\phi(x)\) for \((x,y) \in \Omega \times Y^\bblm\). Then,      \(\ffi\in
L^{p'}(\Omega;C_\#(Y^\bblm;\RRk))\), and by \eqref{eq:lim2sc} we
have%
\begin{equation*}
\begin{aligned}
\lim_{\epsi\to0^+}\int_\Omega u_\epsi(x)\cdot \phi(x)\,\dx
&=\lim_{\epsi\to0^+}\int_\Omega u_\epsi(x)\cdot \ffi\Big(x, \frac{\bR
x}{\epsi} \Big)\,\dx\\ &= \int_\Omega \fint_{Y^\bblm} u(x,y)\cdot
\ffi
(x,y)\,\dx\dy =  \int_{\Omega} \bigg( \fint_{Y^\bblm} u(x,y)\,\dy
\bigg) \cdot\phi(x)\,\dx,
\end{aligned}
\end{equation*}
and this concludes the proof. 
\end{proof}

\begin{proposition}\label{prop:R2scSc}
Let  \(\{u_\epsi\}_{\epsi}
\subset L^p(\Omega;\RRk)\) and \(u\in L^p(\Omega;\RRk)\) be
such that \(u_\epsi\to
u\) in \(L^p(\Omega;\RRk)\) as \(\epsi\to 0^+\). Then,    \(u_\epsi\Rsc{\bR\text{-}2sc}
u \).
\end{proposition}

\begin{proof}
Let \(\ffi
\in L^{p'}(\Omega;C_\#(Y^\bblm;\RRk))\).  Using H\"older's inequality, the convergence  \(u_\epsi\to
u\) in \(L^p(\Omega;\RRk),\) and  Proposition~\ref{prop:Radmissible}
applied to \(\psi(x,y):= u(x)\cdot  \ffi(x,y)\), we get%
\begin{align*}
&\limsup_{\epsi \to0^+} \bigg| \int_\Omega u_\epsi(x)\cdot \ffi\Big(x,
\frac{\bR x}{\epsi}
\Big)\,\dx - \int_\Omega \fint_{Y^\bblm} u(x)\cdot \ffi(x,y)\,\dx\dy\bigg|
\\
&\quad  \leq \limsup_{\epsi \to0^+} \bigg(\Vert u_\epsi - 
u\Vert_{L^p(\Omega;\RRk)} \Vert \ffi\Vert_{L^{p'}(\Omega; C_\#(Y^\bblm;\RRk))}
\\&\hskip21.6mm+ \bigg| \int_\Omega u(x)\cdot \ffi\Big(x, \frac{\bR
x}{\epsi}
\Big)\,\dx - \int_\Omega \fint_{Y^\bblm} u(x)\cdot \ffi(x,y)\,\dx\dy\bigg|\bigg)
\\
&\quad =0. \qedhere
\end{align*}
\end{proof}

The proof of the following version of Riemann--Lebesgue's lemma may be found in \cite[Lemma~2.4]{BoGuZo10}.
This lemma will be used in the subsequent proposition, which
encodes non-trivial examples of sequences that \(\bR\)-two-scale
converge, and will be useful to prove compactness
of bounded sequences in \(L^p(\Omega;\RRk)\) with respect to the
 \(\bR\)-two-scale convergence.

\begin{lemma}[{c.f. \cite[Lemma~2.4]{BoGuZo10}}]\label{lem:RLforR}
Let \(\phi\in C_\#(Y^\bblm;\RRk)\), and assume that \(\bR\)
satisfies 
\eqref{eq:Rcriterion}.
Then, the sequence \(\{\phi_\epsi\}_\epsi \subset L^\infty(\Omega;\RRk)\)
defined by \(\phi_\epsi(x):= \phi\big(\frac{\bR x}{\epsi}\big)\),
 \(x\in\Omega\), converges weakly-\(\star\) in \(L^\infty(\Omega;\RRk)\)
to the constant function \(\bar\phi:=\fint_{Y^\bblm} \phi(y)\,\dy\).
\end{lemma}

\begin{proposition}\label{prop:Radmissible}
Let \(\psi \in L^1(\Omega;C_\#(Y^\bblm;\RRk))\), and assume that \(\bR\)
satisfies 
\eqref{eq:Rcriterion}. Then \(\big\{\psi\big(\cdot, \frac{\bR\,\cdot}{\epsi}\big)\big\}_\epsi\)
is an equiintegrable sequence in \(L^1(\Omega;\RRk)\) such that
\begin{equation}
\label{eq:boundRadmissible}
\begin{aligned}
\Big\Vert \psi\Big(\cdot, \frac{\bR\,\cdot}{\epsi}\Big)\Big\Vert_{L^1(\Omega;\RRk)} \leq
\Vert \psi\Vert_{L^1(\Omega;C_\#(Y^\bblm;\RRk))} = \int_\Omega \sup_{y\in Y^\bblm} |\psi(x,y)|\,\dx
\end{aligned}
\end{equation}
and
\begin{equation}
\label{eq:limRadmissible}
\begin{aligned}
\lim_{\epsi\to0^+} \int_\Omega \psi\Big(x, \frac{\bR x}{\epsi}\Big)\,\dx =
\int_\Omega \fint_{Y^\bblm} \psi(x,y)\,\dx\dy.
\end{aligned}
\end{equation}
\end{proposition}

\begin{proof}
The proof of \eqref{eq:boundRadmissible} is immediate. Using this estimate (that holds with \(\Omega \)  replaced by any
measurable set) and the integrability of the map \(x\in\Omega\mapsto \sup_{y\in
Y^\bblm} |\psi(x,y)|\), we conclude that \(\big\{\psi\big(\cdot, \frac{\bR(\cdot)}{\epsi}\big)\big\}_\epsi\)
is  equiintegrable  in \(L^1(\Omega;\RRk)\). Finally, the proof
of \eqref{eq:limRadmissible}
follows along the lines of that of \cite[Lemma~2.5]{LuNgWa02},
which we detail next.

\textit{Step~1.} Assume that \(\psi \) is of the form \(\psi(x,y)
= \ffi(x) \phi(y)\) with \(\ffi \in L^1(\Omega)\) and  \(\phi\in C_\#(Y^\bblm;\RRk)\). Then, \eqref{eq:limRadmissible} follows from Lemma~\ref{lem:RLforR}.

\textit{Step~2.} Assume that \(\psi \) is of the form \(\psi(x,y)
=\sum_{k=1}^j c_k\chi_{A_k}(x) \phi_k(y)\), where \(j\in\NN\),
\(c_k\) are distinct real numbers, \(A_k\) are mutually disjoint
measurable subsets of \(\Omega\), and    \(\phi_k\in
C_\#(Y^\bblm;\RRk)\). Then, \eqref{eq:limRadmissible} follows from Step~1.

\textit{Step~3.} Let \(\psi \in L^1(\Omega;C_\#(Y^\bblm;\RRk))\). We can find a sequence \(\{\psi_j\}_{j\in\NN}\) of step functions
as in Step~2 such that \(\psi_j \to \psi \) in  \(L^1(\Omega;C_\#(Y^\bblm;\RRk))\)
as \(j\to\infty\). Fix \(j\in\NN\); in view of  \eqref{eq:boundRadmissible}, we have
\begin{equation*}
\begin{aligned}
& \bigg| \int_\Omega \psi\Big(x, \frac{\bR x}{\epsi}\Big)\,\dx
- \int_\Omega \fint_{Y^\bblm} \psi(x,y)\,\dx\dy \bigg| \\
&\quad \leq   \int_\Omega \bigg|\psi\Big(x, \frac{\bR x}{\epsi}\Big)- \psi_j\Big(x,
\frac{\bR x}{\epsi}\Big)\bigg|\,\dx + \bigg| \int_\Omega \psi_j\Big(x,
\frac{\bR x}{\epsi}\Big)\,\dx
- \int_\Omega \fint_{Y^\bblm} \psi_j(x,y)\,\dx\dy \bigg|  \\
&\qquad +  \int_\Omega \fint_{Y^\bblm} \big| \psi_j(x,y)- \psi(x,y)
\big| \,\dx\dy \\
& \quad \leq \big(1+[{\mathcal{L}^\bblm(Y^\bblm)]^{-1}}\big)\Vert \psi - \psi_j\Vert_{L^1(\Omega;C_\#(Y^\bblm;\RRk))}
+ \bigg| \int_\Omega \psi_j\Big(x,
\frac{\bR x}{\epsi}\Big)\,\dx
- \int_\Omega \fint_{Y^\bblm} \psi_j(x,y)\,\dx\dy \bigg|.
\end{aligned}
\end{equation*}
Letting \(\epsi\to 0^+\) and using Step~2  first, and then letting
\(j\to \infty\), we obtain \eqref{eq:limRadmissible} from the
convergence    \(\psi_j \to \psi \) in  \(L^1(\Omega;C_\#(Y^\bblm;\RRk))\)
as \(j\to\infty\).
 \end{proof}
 
 \begin{corollary}\label{cor:R2sc}
 Let \(\psi \in L^p(\Omega;C_\#(Y^\bblm;\RRk))\), and assume that \(\bR\)
satisfies 
\eqref{eq:Rcriterion}. Then, \(\big\{\psi\big(\cdot, \frac{\bR\,\cdot}{\epsi}\big)\big\}_\epsi\)
is a $p$-equiintegrable sequence in \(L^p(\Omega;\RRk)\) that \(\bR\)-two-scale converges to \(\psi\).
\end{corollary}

\begin{proof}
The $p$-equiintegrability assertion follows from Proposition~\ref{prop:Radmissible}
applied to \(|\psi|^p\). The  \(\bR\)-two-scale
convergence assertion follows from \eqref{eq:limRadmissible}
with \(\psi\) replaced by \(\psi\ffi\), where \(\ffi \in L^{p'}(\Omega;C_\#(Y^\bblm;\RRk))\) is an arbitrary function. 
\end{proof}

Using the previous proposition, we   establish next a compactness property with respect to the
 \(\bR\)-two-scale convergence.

\begin{proposition}\label{prop:compactenessR2sc}
Let   \(\{u_\epsi\}_{\epsi} \subset
L^p(\Omega;\RRk)\) be a bounded sequence, and assume that \(\bR\)
satisfies 
\eqref{eq:Rcriterion}. Then, there exist a subsequence \(\epsi'\preceq
\epsi\) and  a function \(u\in L^p(\Omega\times Y^\bblm;\RRk)\)
such that \(u_{\epsi'}\Rsc{\bR\text{-}2sc}
u \).
\end{proposition}

\begin{proof}
The proof follows along the lines of that of \cite[Theorem~14]{LuNgWa02}.

 To simplify
the notation, set \(X:=L^{p'}(\Omega;C_\#(Y^\bblm;\RRk))\), and denote
by \(X'\) the dual of \(X\).
Let \(L_\epsi: X \to\RR\) be the linear map defined, for \(\ffi
\in X\), by
\begin{equation*}
\begin{aligned}
L_\epsi(\ffi):= \int_\Omega u_\epsi(x)\cdot \ffi\Big(x,
\frac{\bR x}{\epsi}
\Big)\,\dx.
\end{aligned}
\end{equation*}

By H\"older's inequality, we have \(|L_\epsi(\ffi)|\leq c\Vert
\ffi \Vert_X\), where \(c:=\sup_{\epsi} \Vert u_\epsi\Vert_{L^p(\Omega;\RRk)}
\). Thus, by the Riesz representation theorem, there exists \(U_\epsi
\in X'\) such that \(\langle U_\epsi,\ffi \rangle_{X',X} = L_\epsi(\ffi)
\) for all \(\ffi
\in X\).
Next, we observe that \(X\) is separable and \(\Vert U_\epsi\Vert_{X'}
= \sup_{\ffi
\in X, \Vert\ffi\Vert_X \leq 1} |\langle U_\epsi,\ffi \rangle_{X',X}|\leq
c\). Hence, by the Alaoglu theorem, there exist a subsequence
\(\epsi'\preceq
\epsi\) and  a function  \(U\in X'\) such that \(\lim_{\epsi'\to
0^+}\langle U_{\epsi'},\ffi \rangle_{X',X} = \langle U,\ffi
\rangle_{X',X}\)  for all \(\ffi\in X\). Passing the inequality
\begin{equation*}
\begin{aligned}
|\langle U_{\epsi'},\ffi
\rangle_{X',X} | = |L_{\epsi'}(\ffi)| \leq c \bigg(\int_\Omega
\Big|
 \ffi\Big(x,
\frac{\bR x}{{\epsi'}}
\Big)\Big|^{p'}\,\dx\bigg)^{\frac{1}{p'}}
\end{aligned}
\end{equation*}
to the limit as \(\epsi'\to0^+\), and invoking Proposition~\ref{prop:Radmissible}
applied to \(\psi(x,y):= |\ffi(x,y)|^{p'}\), we obtain
\begin{equation}\label{eq:Ubdd}
\begin{aligned}
|\langle U,\ffi \rangle_{X',X}| \leq c \bigg(\int_{\Omega}\fint_{Y^\bblm}
|\ffi(x,y)|^{p'}\,\dx\dy \bigg)^{\frac{1}{p'}} =c\big[\mathcal{L}^\bblm(Y^\bblm)\big]^{-\frac{1}{p'}}
\Vert \ffi\Vert_{L^{p'}(\Omega\times
Y^\bblm;\RRk)}
\end{aligned}
\end{equation}
for all \(\ffi
\in X\). Finally, using the density of \(X\) in \(L^{p'}(\Omega\times
Y^\bblm;\RRk)\), \(U\) can be continuously extended to \(L^{p'}(\Omega\times
Y^\bblm;\RRk)\) with \eqref{eq:Ubdd} valid for all \(\ffi \in L^{p'}(\Omega\times
Y^\bblm;\RRk)\). Consequently, by the Riesz representation theorem
there exists \(\bar u\in L^p(\Omega\times Y^\bblm;\RRk)\) such that,
for all  \(\ffi \in L^{p'}(\Omega\times
Y^\bblm;\RRk)\), 
\begin{equation*}
\begin{aligned}
\langle U,\ffi \rangle_{X',X} =\int_{\Omega}\int_{Y^\bblm} \bar u(x,y)\cdot\ffi(x,y)\,\dx\dy
.
\end{aligned}
\end{equation*}
In particular, this last identity holds for all \(\ffi\in X\),
from which we conclude the proof by taking \(u:=\mathcal{L}^\bblm(Y^\bblm)\,\bar
u\).
\end{proof}

\begin{remark}\label{rmk:onRcriterion}
As shown in \cite[Remark~2.8]{BoGuZo10}, Proposition~\ref{prop:compactenessR2sc}
may fail if there exists \(k\in\ZZ^\bblm\backslash \{0\}\) such that
\(\bR^* k =0\).
\end{remark}

\subsection{\textit{R}-two-scale limits of \(\Aa\)-free sequences}\label{subs:2scAfree}
In this subsection, we characterize the \(\bR\)-two-scale limits
associated with \(L^p\)-bounded sequences of \(\Aa\)-free vector
fields, as stated in Theorem~\ref{thm:main2}. As we will show, this characterization is intimately related to the  notion of \((\Aa,\Aa_{\bR^*}^y)\)-free
vector fields introduced below.

\begin{definition}[\(\Aa_{\bR^*}\)- and \(\Aa_{\bR^*}^y\)-free
 fields]\label{def:AR*free}
 We say that \(v\in  L^p_\#(Y^\bblm;\RRd)\) is \(\A_{\bR^*}\)-free,
 and write \(\Aa_{\bR^*} v=0\),
 if for all \(\psi\in C^1_\#(Y^\bblm;\RRl)\), we have
\begin{equation}
\begin{aligned}\label{AR*=0}
\int_{Y^\bblm} v(y)\cdot \Aa^*_{\bR}\psi(y)\,\dy =0,
\end{aligned}
\end{equation}
where 
\begin{equation*}
\begin{aligned}
\Aa^*_{\bR} :=-\sum_{i=1}^\bbln \sum_{m=1}^\bblm
(A^{(i)})^T\bR_{mi}
 \frac{\partial}{\partial y_m}
. 
\end{aligned}
\end{equation*}

We say that \(w\in L^p(\Omega;L^p_\#(Y^\bblm;\RRd))\), with \(w=w(x,y)\), is  \(\A_{\bR^*}^y\)-free,
 and write \(\Aa_{\bR^*}^y w =0\), if 
\(\A_{\bR^*} w(x,\cdot) =0\) for a.e.~\(x\in\Omega\).
 \end{definition}

\begin{remark}[On the notion of \(\Aa_{\bR^*}\)-free]\label{rmk:onAR*}
If \(v\in C^1_\#(Y^\bblm;\RRd)\) satisfies \eqref{AR*=0}, then integration
by parts yields 
\begin{equation*}
\begin{aligned}
0&=\int_{Y^\bblm} v(y)\cdot \Aa^*_{\bR}\psi(y)\,\dy =-\int_{Y^\bblm}
v(y) \cdot\sum_{i=1}^\bbln \sum_{m=1}^\bblm
(A^{(i)})^T\bR_{mi}
 \frac{\partial\psi}{\partial y_m}(y)\,\dy\\&= -\int_{Y^\bblm}\sum_{i=1}^\bbln \sum_{m=1}^\bblm \bR^*_{im}A^{(i)}v(y)\cdot\frac{\partial \psi}{\partial y_m} (y)\,\dy =\int_{Y^\bblm}\sum_{i=1}^\bbln
\sum_{m=1}^\bblm \bR^*_{im}A^{(i)}\frac{\partial v}{\partial
y_m} (y)\cdot \psi(y)\,\dy 
\end{aligned}
\end{equation*}
for all \(\psi\in C^1_\#(Y^\bblm;\RRl)\). Thus, \(\Aa_{\bR^*} v=0\) pointwise in \(\RRm\), where \(\Aa_{\bR^*} :=-\sum_{i=1}^\bbln \sum_{m=1}^\bblm
\bR^*_{im}A^{(i)}
 \frac{\partial }{\partial y_m}\).

We observe further that, as a consequence of our analysis in the Appendix
(see Remark~\ref{rmk:relGRAR*}), in the \(\Aa
= \rm curl\) case, in \(\RRn\), we have that  \(v\in  L^p_\#(Y^\bblm;\RRn)\) is \(\A_{\bR^*}\)-free
if and only if \(v\in \Gg_{\bR}^p\), where \(\Gg_{\bR}^p\) is
given by \eqref{eq:spaceGR}.
\end{remark}

\begin{definition}[\((\Aa,\Aa_{\bR^*}^y)\)-free
 fields]\label{def:AAR*free}
 Let \(w\in L^p(\Omega; L^p_\#(Y^\bblm;\RRd))\), and define \(\bar
w_0\in L^p(\Omega;\RRd)\) and \( \bar w_1 \in L^p(\Omega; L^p_\#(Y^\bblm;\RRd)) \) by setting \(\bar
 w_0:=\int_{Y^\bblm} w(\cdot,y)\,\dy\) and \(\bar w_1:= w - \bar w_0\). We
 say that \(w\) is \((\Aa,\Aa_{\bR^*}^y)\)-free if
\begin{equation*}
\begin{aligned}
\Aa \bar w_0 =0\quad \text{and} \quad \Aa_{\bR^*}^y \bar w_1 =0
\end{aligned}
\end{equation*}
in the sense of Definition~\ref{def:Afree} and Definition~\ref{def:AR*free}, respectively.
\end{definition}

The next proposition shows that the  \(\bR\)-two-scale limit
of an \(L^p\)-bounded sequence of \(\Aa\)-free vector
fields is necessarily \((\Aa, \Aa_{\bR^*}^y)\)-free.

\begin{proposition}\label{prop:necR2sc}
Let \(\{u_\epsi\}_\epsi\) be a bounded and \(\Aa\)-free sequence
in \(L^p(\Omega;\RRd)\). Assume  that there exists a function \(u\in L^p(\Omega\times Y^\bblm;\RRd)\) such
that \(u_{\epsi}\Rsc{\bR\text{-}2sc}
u \).
Then,  \(u\) is \((\Aa,\Aa_{\bR^*}^y)\)-free in the sense of
Definition~\ref{def:AAR*free}.
\end{proposition}

\begin{proof}
Let \(\phi\in C^1_c(\Omega;\RRl)\). Using   the fact
that each \(u_\epsi\) is \(\Aa\)-free first, and  invoking \eqref{eq:lim2sc}
applied to \(\varphi:=\Aa^*\phi\),
we get
\begin{equation}
\begin{aligned}\label{eq:R2scsuf1}
0=\lim_{\epsi\to0^+} \int_\Omega u_\epsi(x)\cdot \Aa^*\phi (x)\,\dx
= \int_\Omega \fint_{Y^\bblm} u(x,y)\cdot  \Aa^*\phi (x)\,\dx \dy =
\int_\Omega \bar u_0(x)\cdot  \Aa^*\phi (x)\,\dx
\end{aligned}
\end{equation}
where \(\bar u_0:= \int_{Y^\bblm} u(\cdot,y)\,\dy\). Recalling
Definition~\ref{def:Afree}, \eqref{eq:R2scsuf1} shows that \(\Aa
\bar u_0 =0\) in \(L^p(\Omega;\RRl)\). 

Next, we prove that \(\Aa_{\bR^*}^y \bar u_1
=0\) with \(\bar u_1:= u - \bar u_0\).  Let \(\phi\in C^1_c(\Omega)\)
and \(\psi\in C^1_\#(Y^\bblm;\RRl)\), and   set  \(\varphi_\epsi(x):=
\epsi\phi(x)\psi\big( \frac{\bR x}{\epsi}\big)  \) for \(x\in\Omega\).
Then \(\varphi_\epsi \in C^1_c(\Omega;\RRl)\) with
\begin{equation*}
\begin{aligned}
\Aa^*\varphi_\epsi (x) &=-\sum_{i=1}^\bbln (A^{(i)})^T\frac{\partial
\varphi_\epsi}{\partial x_i}(x)\\&= -\epsi \sum_{i=1}^\bbln \frac{\partial
\phi}{\partial x_i}(x)(A^{(i)})^T
 \psi\Big(\frac{\bR x}{\epsi}\Big) - \phi(x)\sum_{i=1}^\bbln \bigg[(A^{(i)})^T\sum_{m=1}^\bblm
 \frac{\partial \psi}{\partial y_m}\Big( \frac{\bR x}{\epsi}\Big) \bR_{mi}\bigg]\\
&= -\epsi \sum_{i=1}^\bbln \frac{\partial
\phi}{\partial x_i}(x)(A^{(i)})^T
 \psi\Big(\frac{\bR x}{\epsi}\Big) - \phi(x)\sum_{i=1}^\bbln \sum_{m=1}^\bblm(A^{(i)})^T\bR_{mi}
 \frac{\partial \psi}{\partial y_m}\Big( \frac{\bR x}{\epsi}\Big)
.
\end{aligned}
\end{equation*}
%
Hence, arguing as above, we have
\begin{equation}
\label{eq:R2scsuf2}
\begin{aligned}
0&=\lim_{\epsi\to0^+} \int_\Omega u_\epsi(x) \Aa^* \varphi_\epsi(x)\,\dx
\\&= \lim_{\epsi\to0^+} \int_\Omega u_\epsi(x) \cdot \bigg( -\epsi\sum_{i=1}^\bbln
\frac{\partial
\phi}{\partial x_i}(x)(A^{(i)})^T
 \psi\Big(\frac{\bR x}{\epsi}\Big) - \phi(x)\sum_{i=1}^\bbln \sum_{m=1}^\bblm(A^{(i)})^T\bR_{mi}
 \frac{\partial \psi}{\partial y_m}\Big( \frac{\bR x}{\epsi}\Big)
\bigg)\,\dx \\
& =-\int_{\Omega\times Y^\bblm} u(x,y)\cdot \bigg(\phi(x)\sum_{i=1}^\bbln \sum_{m=1}^\bblm(A^{(i)})^T\bR_{mi}
 \frac{\partial \psi}{\partial y_m}(y)\bigg)\,\dx\dy\\
&= -\int_\Omega \fint_{Y^\bblm} u(x,y) \cdot \phi(x)
\Aa^*_{\bR}\psi(y)\,\dx\dy= -\int_\Omega \fint_{Y^\bblm} \bar u_1(x,y) \cdot \phi(x)
\Aa^*_{\bR}\psi(y)\,\dx\dy,
\end{aligned}
\end{equation}
where in the last equality  we used the fact that \(\bar u_0\)
depends only on \(x\) and  \(\int_{Y^\bblm} 
\Aa^*_{\bR}\psi(y)  \,\dy=0\) by the periodicity of \(\psi\).

Because \eqref{eq:R2scsuf2} holds for all \(\phi\in C^1_c(\Omega)\)
and \(\psi\in C^1_\#(Y^\bblm;\RRl)\) and \(C^1_\#(Y^\bblm;\RRl)\) is
separable, we conclude that \(\Aa_{\bR^*}^y \bar u_1
=0\) in the sense of  Definition~\ref{def:AR*free}. 
\end{proof}

The next proposition shows that Proposition~\ref{prop:necR2sc} fully characterizes the  \(\bR\)-two-scale limit
of an \(L^p\)-bounded sequence of \(\Aa\)-free vector
fields, as we prove that any  \((\Aa, \Aa_{\bR^*}^y)\)-free vector
field is attained as the \(\bR\)-two-scale limit of such a sequence. As we mentioned in the Introduction, this result is new in the literature even for \(p=2\) and \(\Aa=\rm
curl\) or \(\Aa=\rm div\) which were treated in  \cite{BoGuZo10, WeGuCh19}. 

\begin{proposition}\label{prop:sufR2sc}
Let \(u\in L^p(\Omega;L_\#^p(Y^\bblm;\RRd)) \) be a \((\Aa,\Aa_{\bR^*}^y)\)-free
vector field in the sense of
Definition~\ref{def:AAR*free}, and assume that \(\bR\)
satisfies 
\eqref{eq:Rcriterion}.
 Then, there exists a  bounded and
\(\Aa\)-free sequence, \(\{u_\epsi\}_\epsi\), 
in \(L^p(\Omega;\RRd)\) such
that \(u_{\epsi}\Rsc{\bR\text{-}2sc}
u \).
\end{proposition}

\begin{proof}
Fix  
 \(u\in L^p(\Omega; L^p_\#(Y^\bblm;\RRd))\), a  \((\Aa,\Aa_{\bR^*}^y)\)-free
vector field. We have  
\begin{equation}\label{eq:Au0AR*u1}
\begin{aligned}
\Aa \bar u_0 =0\quad \text{and} \quad \Aa_{\bR^*}^y \bar u_1
=0
\end{aligned}
\end{equation}
in the sense of Definition~\ref{def:Afree} and Definition~\ref{def:AR*free},
respectively, where \(\bar
 u_0:=\int_{Y^\bblm} u(\cdot,y)\,\dy\) and \(\bar u_1:= u - \bar
u_0\). Note that for a.e.~\(x\in\Omega\), it holds
\begin{equation}\label{eq:meanu1}
\begin{aligned}
\int_{Y^\bblm} \bar u_1(x,y)\,\dy =0.
\end{aligned}
\end{equation}

 We will proceed in three steps.  

\textit{Step~1.} Assume that \(\bar u_0=0\) and \(\bar u_1\in C^1_c(\RRn;
C^1_\#(Y^\bblm;\RRd))\). In this case, \eqref{eq:meanu1} holds for
all \(x\in\Omega\) and, as observed in Remark~\ref{rmk:onAR*},
we have
\begin{equation}\label{eq:AR*u1=0}
\begin{aligned}
\Aa_{\bR^*} \bar u_1=0
\text{ pointwise in } \Omega\times\RRm, \text{ where } \Aa_{\bR^*} =-\sum_{i=1}^\bbln \sum_{m=1}^\bblm
\bR^*_{im}A^{(i)}
 \frac{\partial }{\partial y_m}.
\end{aligned}
\end{equation}

For each \(\epsi>0\), define \(v_\epsi\in C^1_c(\RRn)\) by setting
\begin{equation}\label{eq:defvepsi}
\begin{aligned}
v_\epsi(x):= \bar u_1\Big(x,\frac{\bR x}{\epsi}\Big)\enspace \text{
for } x\in\RRn.
\end{aligned}
\end{equation}
By Corollary~\ref{cor:R2sc} and Proposition~\ref{prop:R2sc-wc},
together with \eqref{eq:meanu1}, we obtain
\begin{equation}\label{eq:recR2sc1}
\begin{aligned}
& \{v_\epsi\}_\epsi \text{ is a \(p\)-integrable sequence in
\(L^p(\Omega;\RRd)\)},\\
&  v_{\epsi}\Rsc{\bR\text{-}2sc}
\bar u_1, \\
& v_\epsi \weakly0 \text{ weakly in \(L^p(\Omega;\RRd)\)}.
\end{aligned}
\end{equation}

On the other hand, in view of \eqref{eq:AR*u1=0} and recalling
Remark~\ref{rmk:Aseveralv}, we have
\begin{equation*}
\begin{aligned}
\Aa v_\epsi (x) = (\Aa_x \bar u_1) \Big(x,\frac{\bR x}{\epsi}\Big)
\end{aligned}
\end{equation*}
for all \(x\in\Omega\). Because \(\Aa_x \bar u_1 \in C_c(\RRn;C^1_\#(Y^\bblm;\RRl))\),
we may invoke Proposition~\ref{prop:R2sc-wc}
and \eqref{eq:meanu1} once more to conclude that
\begin{equation*}
\begin{aligned}
\Aa v_\epsi \weakly \int_{Y^\bblm} (\Aa_x \bar u_1) (x,y) \,\dy =
\Aa_x \bigg(\int_{Y^\bblm} \bar u_1(x,y)\,\dy\bigg) =0 \text{ weakly
in \(L^p(\Omega;\RRl)\)}.
\end{aligned}
\end{equation*}
Hence, 
\begin{equation}
\label{eq:eq:recR2sc2}
\begin{aligned}
\Aa v_\epsi \to 0 \text{ in \(W^{-1,p}(\Omega;\RRl)\)}.
\end{aligned}
\end{equation}

Let \(\Pi\subset\RRn\) be a parallelotope containing \(\Omega\).
By Lemma~\ref{lem:2.8FK}, we can find a \(p\)-equiintegrable
sequence in \(\Pi\), \(\{u_\epsi\}_\epsi\subset L^p(\Pi;\RRd)\), and
a positive constant depending only on \(\Aa\),  \(C=C(\Aa)\), such that
\begin{equation}
\label{eq:seqbyextlem}
\begin{aligned}
& \Aa u_\epsi = 0 \mbox{
for all } \epsi>0,\\
&u_\epsi - v_\epsi \to 0 \text{ in } L^p(\Omega;\RRd),\\
&\|u_\epsi\|_{L^p(\Pi;\RRd)} \leqslant C \|v_\epsi\|_{L^p(\Omega;\RRd)}
\mbox{
for all } \epsi>0 .
\end{aligned}
\end{equation}
To conclude Step~1, we observe that the second condition in \eqref{eq:recR2sc1}
and \eqref{eq:seqbyextlem}, together with Proposition~\ref{prop:R2scSc}, yield
\begin{equation}\label{eq:seqbyextlem1}
\begin{aligned}
u_{\epsi}\Rsc{\bR\text{-}2sc}
\bar u_1.
\end{aligned}
\end{equation}

\textit{Step~2.} Assume that \(\bar u_0=0\) and \(\bar u_1\in
L^p(\Omega;L_\#^p(Y^\bblm;\RRd))\).

For all \(y\in\RRm \), we extend \(\bar u_1(\cdot,y)\) by zero outside \(\Omega\), which  we still denote by \(\bar u_1\). Let \(\{\rho_j\}_{j\in\NN}
\subset C_c^\infty(\RRn)\) and \(\{\rho_j^\#\}_{j\in\NN}
\subset C_\#^\infty(Y^\bblm)\) be  sequences of standard, symmetric, mollifiers. For each \(j\in\NN\), we define
\begin{equation*}
\begin{aligned}
\tilde u_j(x,y):=& \int_{\RRn}\!\int_{Y^\bblm} \bar u_1(x',y') \rho_j(x-x')
\rho^\#_j (y-y')\, \dx'\dy'
\\
 = & \int_{\RRn}\!\int_{Y^\bblm} \bar u_1(x',y'+y) \rho_j(x-x')
\rho^\#_j (y')\, \,\dx'\dy,
\end{aligned}
\end{equation*}
where in the last equality we used the \(Y^\bblm\)-periodicity of \(\bar
u_1\) along with the symmetry and \(Y^\bblm\)-periodicity of  \(\rho_j^\#\).
By standard mollification arguments, we have \(\tilde u_j \in C^\infty_c(\RRn;
C^\infty_\#(Y^\bblm;\RRd))\) with 
\begin{equation*}
\begin{aligned}
\Vert \tilde u_j\Vert_{L^p(\RRn;L^p(Y^\bblm:\RRd))} \leq \Vert \bar u_1\Vert_{L^p(\Omega;L^p(Y^\bblm:\RRd))}.
\end{aligned}
\end{equation*}
Moreover,
\begin{equation*}
\begin{aligned}
\Aa_{\bR^*} \tilde u_j=0
\text{ pointwise in } \Omega\times\RRm \enspace \text{ and }
\enspace \int_{Y^\bblm} \tilde u_j (\cdot, y)\,\dy =0
\end{aligned}
\end{equation*}
by \eqref{eq:Au0AR*u1} and  \eqref{eq:meanu1}, together with the \(Y^\bblm\)-periodicity
of \(\bar u_1\) and Fubini's theorem.

By Step~1, for each \(j\in\NN\), we can find
a \(p\)-equiintegrable
sequence in \(\Pi\), \(\{u_\epsi^{(j)}\}\subset L^p(\Pi;\RRd)\), 
satisfying \eqref{eq:seqbyextlem}--\eqref{eq:seqbyextlem1}
with \(u_\epsi\) replaced by \(u_\epsi^{(j)}\) and, recalling
\eqref{eq:defvepsi}, \(v_\epsi\) replaced by
\begin{equation*}
\begin{aligned}
v_\epsi^{(j)}(x):= \tilde u_j\Big(x,\frac{\bR x}{\epsi}\Big)\enspace
\text{
for } x\in\RRn.
\end{aligned}
\end{equation*}
In particular, we have
\begin{equation*}
\begin{aligned}
\limsup_{j\to\infty} \limsup_{\epsi\to0^+} \int_{\Omega} \big|u_\epsi^{(j)}(x)\big|^p\,
\dx& \leq C^p\limsup_{j\to\infty} \limsup_{\epsi\to0^+} \int_{\Omega}
\Big| \tilde u_j\Big(x,\frac{\bR x}{\epsi}\Big) \Big|^p\,\dx\\&=
C^p\limsup_{j\to\infty} \int_{\Omega\times Y^\bblm} |\tilde u_j(x,y)|^p
\, \dx\dy\\
&\leq C^p \Vert
\bar u_1\Vert_{L^p(\Omega;L^p(Y^\bblm:\RRd))}.
\end{aligned}
\end{equation*}
This estimate and the separability of \(L^{p'}(\Omega; C_\#(Y^\bblm;\RRd))\)
allow us to use a diagonalization argument as
in \cite[proof of Proposition 1.11 (p.449)]{FeFo120} to find a sequence
\((j_\epsi)_\epsi\) such that \(j_\epsi\to\infty\) as \(\epsi\to0^+\)
and \(u_\epsi:= u_\epsi^{(j_\epsi)}\) satisfies the required
properties.

\textit{Step~3.} We treat the general case.

By Step~2, there exists a  bounded and
\(\Aa\)-free sequence, \(\{ u_\epsi\}_\epsi\), 
in \(L^p(\Omega;\RRd)\) such
that \(u_{\epsi}\Rsc{\bR\text{-}2sc}
u_1 \). Defining
\(\hat u_\epsi := u_0 + u_\epsi \), we have \(\Aa \hat u_\epsi =0 \) and  \(\hat u_{\epsi}\Rsc{\bR\text{-}2sc}
u_0 + u_1 =u \), using   \eqref{eq:Au0AR*u1} and  Proposition~\ref{prop:R2scSc}.
\end{proof}

\begin{proof}[Proof of Theorem~\ref{thm:main2}] The statement
in Theorem~\ref{thm:main2} in an immediate consequence of Propositions~\ref{prop:necR2sc} and \ref{prop:sufR2sc}.
\end{proof}


\section{$\Gamma$-convergence homogenization}\label{sect:proofmain}
 
 In this section, we prove Theorem~\ref{thm:main}. To this end,
 we first show in Theorem~\ref{t.mt} below that the sequence \(\{F_\epsi\}_\epsi\), with
 \(F_\epsi\) given by  \eqref{eq:defFe}, \(\Gamma\)-converges
 to  a certain functional, \(\mathcal{F}_{\mathrm{hom}}\), with respect to the weak topology in
 \(L^p(\Omega;\RRd)\), as \(\epsi\to0^+\).
Then, in Proposition~\ref{prop:intrep} below, we establish the integral representation
for this \(\Gamma\)-limit as stated in Theorem~\ref{thm:main}.

\begin{theorem} \label{t.mt}
Let \(\Omega\subset\RRn\) be
an open and bounded set,  let 
 $f_R: \Omega \times \RRn \times \RRd
\to [0,\infty)$ be a function satisfying (H1)--(H3), let \(F_\epsi\)
be the functional introduced in \eqref{eq:defFe}, and
 assume
that \eqref{constrank} holds. Then, the sequence \(\{F_\epsi\}_\epsi\)  \(\Gamma\)-converges on \(\su = \big\{u \in L^p(\Omega;\RRd)\!: \, \Aa u = 0\big \}\) as \(\epsi\to0^+\), with respect to
the weak topology in
 \(L^p(\Omega;\RRd)\),    to the functional \(\mathcal{F}_{\mathrm{hom}}\) defined, for \(u\in  \su\), by
\begin{equation*}
\begin{aligned}
\mathcal{F}_{\mathrm{hom}}(u):=\inf_{w\in \sw} \int_\Omega \fint_{Y^\bblm} f(x,y,u(x)
+ w(x,y))\,\dx\dy,
\end{aligned}
\end{equation*}
where
\begin{equation}
\begin{aligned}\label{eq:setWA}
    \sw := \bigg\{w \in L^p(\Omega; L^p_\# (Y^\bblm; \RRd))\! : &\,w \mbox{
is } (\Aa, \Aa^y_{\bR^*})\mbox{-free in the sense of Definition~\ref{def:AAR*free},}\\
&\, \text{with}\int_{Y^\bblm} w(\cdot, y)\,\dy=0\bigg\}.
\end{aligned}
\end{equation}
Precisely, given an arbitrary sequence $\{\e_n\}_{n \in \mathbb{N}}
\subset \R^+$  converging to $0$,  the following pair of statements holds: 
\begin{enumerate}
    \item (\(\Gamma\)-liminf inequality) Let $\{ u_n\}_{n \in \mathbb{N}} \subset \su$ be a sequence such that $u_n \weakly u$ in \(L^p(\Omega;\RRd)\) for some \(u\in L^p(\Omega;\RRd)\). Then, \(u\in\su\) and
    \begin{align*}
        \liminf_{n\to\infty} F_{\e_n}(u_n) \geqslant \mathcal{F}_{\mathrm{hom}}(u).
    \end{align*}
    \item (recovery sequence) For every $u \in \su$,  there
exists sequence $\{u_n\}_{n\in\NN} \subset \su$ such that $u_n \rightharpoonup
u$  in $L^p(\Omega;\RRd)$  and
    \begin{align*}
         \limsup_{n\to\infty} F_{\e_n}(u_n) \leqslant \mathcal{F}_{\mathrm{hom}}(u).
    \end{align*}
\end{enumerate}
\end{theorem}

The proof of  Theorem \ref{t.mt} is obtained as a consequence of  Propositions \ref{prop:upperb} and \ref{prop.lbnd} below. We begin with a lemma that will be used in the subsequent proposition,
where we establish the \textit{recovery sequence} property, and is
a simple adaptation of \cite[Proposition~3.5-(i)]{FoKr10}.
\begin{lemma}\label{lem:Prop3.5i}
Assume that  hypotheses (H1)--(H2) hold.  Let \(\{\epsi_n\}_{n\in\NN}\subset \RR^+\) be a sequence
converging to \(0\), and let \(\{u_n\}_{n\in\NN}\), \(\{w_n\}_{n\in\NN}
\subset L^p(\Omega;\RRd)\) be two \(p\)-equiintegrable sequences
such that \(\lim_{n\to\infty} \Vert u_n- w_n\Vert_{L^p(\Omega;\RRd)}
=0\). 
Then,
\begin{equation*}
\begin{aligned}
\lim_{n\to\infty} \int_\Omega \bigg[  f_R\Big( x , \frac{x}{\e_n},
u_n(x) \Big) \,\dx- f_R\Big( x ,
\frac{x}{\e_n},
w_n(x) \Big) \bigg] \,\dx =0.
\end{aligned}
\end{equation*}
\end{lemma}

\begin{proof}
Fix \(\tau>0\). We want to show that there  exists \(n_0= n_0(\tau)\in\NN\)
such that if  \(n\geq n_0\), then 
\begin{equation*}
\begin{aligned}
\bigg|\int_\Omega \Big[  f_R\Big( x , \frac{x}{\e_n},
u_n(x) \Big) \,\dx- f_R\Big( x ,
\frac{x}{\e_n},
w_n(x) \Big) \Big] \,\dx \bigg| \leq \tau.
\end{aligned}
\end{equation*}  

Using the \(p\)-equiintegrability of \(\{u_n\}\) and \(\{w_n\}\), there exists \(\delta=\delta(\tau)>0\) such that if \(E\subset\Omega\)
is a measurable set with \(|E|<\delta\), then
\begin{equation}
\label{eq:peqiiunwn}
\begin{aligned}
\sup_{n\in\NN} \int_E C \big (2+|u_n(x)|^p + |w_n(x)|^p\big)\, \dx <\frac{\tau}{8}.
\end{aligned}
\end{equation}
Moreover, there exists \(r_\delta>0\) such that
\begin{equation}
\begin{aligned}
\label{eq:peqiiunwn1}
\sup_{n\in\NN}\Big[\big|\{|u_n|\geq r_\delta\}\big| + \big|\{|w_n|\geq r_\delta\}\big| \Big] \leq \delta.
\end{aligned}
\end{equation}

Let \(\Omega_{\delta} \Subset \Omega\) be such that \(|\Omega\setminus
\Omega_{\delta}| \leq \delta\). Using the continuity assumption
on \(f\) and the \(Y^\bblm\)-periodicity of \(f\) with respect to its second variable,  we conclude that \(f\) is uniformly continuous
on \(\overline\Omega_\delta\times \RRm \times \overline B_{r_\delta}(0) \). Thus, we can find \(0<\bar\delta \leq \delta\)
such that,  for all \(x\in\Omega_\delta\), \(y\in \RRm\), and \(\xi_1\), \(\xi_2\in B_{r_\delta}(0)\) with \(|\xi_1
- \xi_2| \leq \bar \delta\), we have
\begin{equation}
\label{eq:peqiiunwn2}
\begin{aligned}
\big| f(x,y,\xi_1) - f(x,y,\xi_2) \big| \leq \frac{\tau}{2|\Omega_{\delta}|}.
\end{aligned}
\end{equation}

Finally, by Chebyshev's inequality, there exists \(0<\tilde \delta <\bar \delta \)
such that if \(\Vert v\Vert_{L^p(\Omega;\RRd)}<\tilde
\delta\), then 
\begin{equation}
\label{eq:peqiiunwn3}
\begin{aligned}
\big| \{|v|\geq \bar \delta \} \big| \leq \delta.
\end{aligned}
\end{equation}
We observe further that because \(\lim_{n\to\infty} \Vert u_n- w_n\Vert_{L^p(\Omega;\RRd)}
=0\), we can find \(n_0 = n_0(\tau)\in\NN\) such that  \(\Vert u_n- w_n\Vert_{L^p(\Omega;\RRd)}
<\tilde \delta\) for all \(n\geq n_0\). 

Thus, for each \(n\geq n_0\) and for \(A:=
(\Omega\setminus \Omega_\delta) \cup \{|u_n|\geq r_\delta\} \cup \{|w_n|\geq r_\delta\} \cup  \{|u_n - w_n|\geq \bar \delta \} \), we conclude from (H2), (H3), and  \eqref{eq:peqiiunwn}--\eqref{eq:peqiiunwn3} that
\begin{align*}
&\bigg|\int_\Omega \Big[  f_R\Big( x , \frac{x}{\e_n},
u_n(x) \Big) \,\dx- f_R\Big( x ,
\frac{x}{\e_n},
w_n(x) \Big) \Big] \,\dx \bigg| \\ &\quad\leq \int_A C\big(2 + |u_n(x)|^p
+\ |w_n(x)|^p\big) \,\dx + \bigg|\int_{\Omega\setminus A}  \Big[  f\Big( x , \frac{\bR x}{\e_n},
u_n(x) \Big) \,\dx- f\Big( x ,
\frac{\bR x}{\e_n},
w_n(x) \Big) \Big] \,\dx \bigg|\\
&\quad \leq \frac{\tau}{2} + \frac{\tau}{2|\Omega_{\delta}|}|\Omega\setminus
A| \leq \tau.\qedhere
\end{align*}
\end{proof}

The \textit{recovery sequence} property in Theorem~\ref{t.mt} is a simple consequence of the following proposition. We observe that this result  does not require  assumption (H3) to hold. 

\begin{proposition}\label{prop:upperb}
Assume that hypotheses (H1)--(H2) hold,  and let  \(\mathcal{U}_\mathcal{A}\)
be the set introduced in \eqref{eq:setUA}. 
Then, for each   $\delta>0$,  $u \in
\mathcal{U}_\mathcal{A},$ and \(w\in \overline{\mathcal{W}}_\Aa := \big\{w \in L^p(\Omega; L^p_\# (Y^\bblm;
\RRd))\! : \,w \mbox{
is } (\Aa, \Aa^y_{\bR^*})\)-free in the sense of Definition~\ref{def:AAR*free}\(\big\}\), there
exists a sequence $\{u_\e\} \subset \mathcal{U}_\mathcal{A}$ such that
$u_\e \rightharpoonup u + \bar w_0$ weakly in $L^p(\Omega;\RRd)$ as $\e \to 0^+$
 and, for all \(\kappa\in\NN\), 
\begin{align} \label{upbndeta}
    \lim_{\e \to 0^+} \int_\Omega f_R\Big( x , \frac{x}{\e},
u_\e(x) \Big) \,\dx \leqslant \int_\Omega \fint_{Y^\bblm} f(x, \kappa y, u(x) +
w(x,y))\,\dy \dx + \delta, \end{align}
where, recalling Definition~\ref{def:AAR*free}, \(\bar w_0:=\fint_{Y^\bblm} w(\cdot, y)\,\dy\). 
\end{proposition}

\begin{proof}
  Fix   $\delta>0$,  $u \in
\mathcal{U}_\mathcal{A},$  and  $ w \in \overline{\mathcal{W}}_\Aa$.
We will proceed in two steps, first assuming extra regularity
on \(w\),
and then treating the general case.

\textit{Step~1.} Recalling the decomposition \(w= \bar w_0 +
\bar w_1 \) introduced in Definition~\ref{def:AAR*free}, assume
that \(\bar w_1 \in C^1(\overline\Omega; C_\#^1(Y^\bblm;\RRd))\). 

For \(\kappa\in\NN\) and \((x,y)\in \Omega \times Y^\bblm\), define
\begin{equation*}
\begin{aligned}
\psi(x,y):= f(x, \kappa y, u(x) + w(x,y))=f(x, \kappa y, u(x) + \bar w_0(x)
+ \bar w_1(x,y)).
\end{aligned}
\end{equation*}
Using (H1), (H2), the continuity of \(f\), and the regularity of \(\bar w_1\), we conclude that
\(\psi \in L^1(\Omega; C_\#(Y^\bblm))\). Then, by Proposition~\ref{prop:Radmissible},
we have
\begin{equation*}
\begin{aligned}
\lim_{\epsi\to0^+} \int_\Omega \psi\Big(x, \frac{\bR x}{\epsi}\Big)\,\dx
=
\int_{\Omega} \fint_{Y^\bblm} \psi(x,y)\,\dx\dy;
\end{aligned}
\end{equation*}
i.e.,
\begin{equation}\label{eq:onub1}
\begin{aligned}
\lim_{\epsi\to0^+} \int_\Omega f_R\Big( x , \frac{x}{\e},
w_\e(x) \Big) \,\dx =  \int_\Omega \fint_{Y^\bblm} f(x, \kappa y,
u(x) +
w(x,y))\,\dy \dx,
\end{aligned}
\end{equation}
where, for \(x\in\Omega\),
\begin{equation*}
\begin{aligned}
w_\epsi(x):= u(x) + \bar w_0(x) + \bar w_1\Big(x, \frac{\bR x}{\epsi}\Big).
\end{aligned}
\end{equation*}

Arguing as in Step~1 of the proof of Proposition~\ref{prop:sufR2sc}
with \(\bar u_1\) replaced by \(\bar w_1\) in \eqref{eq:defvepsi},
and using the fact that \(\Aa u+\Aa \bar w_0=0\) by the definition
of  \(\mathcal{U}_\mathcal{A}\)
and \(\overline{\mathcal{W}}_\Aa\), we conclude that (see \eqref{eq:recR2sc1}--\eqref{eq:eq:recR2sc2})
\begin{equation*}
\begin{aligned}
& \{w_\epsi\}_\epsi \text{ is a \(p\)-integrable sequence in
\(L^p(\Omega;\RRd)\)},\\
& w_\epsi \weakly u + \bar w_0 \text{ weakly in \(L^p(\Omega;\RRd)\)},\\
&\Aa w_\epsi \to 0 \text{ in \(W^{-1,p}(\Omega;\RRl)\)}.
\end{aligned}
\end{equation*}
Then, by Lemma~\ref{lem:2.8FK}, we can find a sequence \(\{u_\epsi\}_\epsi
\subset L^p(\Omega;\RRd)\) such that
\begin{equation*}
\begin{aligned}
& \{u_\epsi\}_\epsi \text{ is  \(p\)-integrable},\\
& \Aa u_\epsi =0 \text{ in $L^p(\Omega;\RRl)$},\\
& u_\epsi - w_\epsi \to 0 \text{ in \(L^p(\Omega;\RRd)\)}.
\end{aligned}
\end{equation*}
In particular, 
$u_\e \rightharpoonup u + \bar w_0$ weakly in $L^p(\Omega;\RRd)$. Moreover, by Lemma~\ref{lem:Prop3.5i}, we have
\begin{equation*}
\begin{aligned}
\lim_{\epsi\to0^+} \int_\Omega f_R\Big( x , \frac{x}{\e},
u_\e(x) \Big) \,\dx = \lim_{\epsi\to0^+} \int_\Omega f_R\Big( x , \frac{x}{\e},
w_\e(x) \Big) \,\dx,
\end{aligned}
\end{equation*}
which, together with \eqref{eq:onub1}, concludes Step~1.

\textit{Step~2.} We treat the general case.

Fix \(j\in\NN\). Arguing as in Step~1 of the proof of Proposition~\ref{prop:sufR2sc}
with \(\bar u_1\) replaced by \(\bar w_1\), we can find \(\tilde
w_j \in C^1(\overline\Omega; C_\#^1(Y^\bblm;\RRd))\) such
that \(\bar w_0 + \tilde w_j \in\overline{\mathcal{W}}_\Aa\)
and \(\Vert \tilde w_j - \bar w_1\Vert_{L^p(\Omega; L^p_\#(Y^\bblm;\RRd))}
\leq \frac1j\). Then, extracting a subsequence of \(\{\tilde
w_j\}_{j\in\NN}\) if necessary, Vitali--Lebesgue  theorem and (H1)--(H2) yield
\begin{equation*}
\begin{aligned}
&\lim_{j\to\infty}  \int_\Omega \fint_{Y^\bblm} f(x, \kappa y,
u(x) + 
 \bar w_0(x)
+ \tilde w_j(x,y))\,\dy \dx \\&\quad= \int_\Omega \fint_{Y^\bblm} f(x, \kappa y,
u(x) + 
 \bar w_0(x)
+ \bar w_1(x,y))\,\dy \dx.
\end{aligned}
\end{equation*}
Hence, we can find \(j_\delta\in\NN\) such that 
\begin{equation*}
\begin{aligned}
&\int_\Omega \fint_{Y^\bblm} f(x, \kappa y,
u(x) + 
 \bar w_0(x)
+ \tilde w_{j_\delta}(x,y))\,\dy \dx\\&\quad \leq \int_\Omega \fint_{Y^\bblm} f(x, \kappa y,
u(x) + 
 \bar w_0(x)
+ \bar w_1(x,y))\,\dy \dx + \delta.
\end{aligned}
\end{equation*}

To conclude, we invoke  Step~1 to find
 a sequence $\{u_\e\} \subset \mathcal{U}_\mathcal{A}$,
that depends on \(\delta\), \(u\), and \(w\),  
such that
$u_\e \rightharpoonup u + \bar w_0$ weakly in $L^p(\Omega;\RRd)$
as $\e \to 0^+$
 and, for all \(\kappa\in\NN\), 
\begin{equation*} 
    \lim_{\e \to 0^+} \int_\Omega f_R\Big( x , \frac{x}{\e},
u_\e(x) \Big) \,\dx = \int_\Omega \fint_{Y^\bblm} f(x, \kappa y,
u(x) + 
 \bar w_0(x)
+ \tilde w_{j_\delta}(x,y))\,\dy \dx.\qedhere
  \end{equation*}
\end{proof}

Next, we establish the \(\Gamma\)-liminf inequality property stated in Theorem~\ref{t.mt}.

\begin{proposition} \label{prop.lbnd}
Let \(\{\epsi_n\}_{n\in\NN}\subset
\RR^+\) be a sequence
converging to \(0\), and let $\{ u_n\}_{n \in \mathbb{N}}
\subset \su$ be a sequence such that $u_n \weakly u$ in \(L^p(\Omega;\RRd)\)
for some \(u\in L^p(\Omega;\RRd)\). Then, under the assumptions of Theorem~\ref{t.mt}, we have  \(u\in\su\) and
\begin{equation}
\label{eq:lowerb}
\begin{aligned}
         \liminf_{n\to\infty} F_{\e_n}(u_n) \geqslant \mathcal{F}_{\mathrm{hom}}(u).
\end{aligned}
\end{equation}
\end{proposition}

\begin{proof}

The condition 
 \(u\in\su\) follows from the fact that \(u_n\in \su\) for all \(n\in\NN\) together with the convergence  $u_n \weakly u$ in \(L^p(\Omega;\RRd)\).
Moreover, by Propositions~\ref{prop:necR2sc} and \ref{prop:R2sc-wc}, and  the uniqueness of \(\bR\)-two-scale limits (see Remark~\ref{rmk:uniq}),  we have   \(u_{n}\Rsc{\bR\text{-}2sc}
v \) for a vector-field  \(v\) that  is \((\Aa,\Aa_{\bR^*}^y)\)-free in the sense
of
Definition~\ref{def:AAR*free}, with \(\int_{Y^\bblm} v(\cdot, y)\,\dy = u(\cdot)\). In particular, we have the decomposition 
\begin{align*}
    v= u+ v_1, \quad  v_1 \in L^p(\Omega;L^p_\#
(Y^\bblm;\RRd)), \quad \sA^y_{\bR^*} v_1 = 0,\quad  \int_{Y^\bblm}
v_1(\cdot,y)\,\dy = 0. 
\end{align*}

Let \(\{\psi_j\}_{j\in\NN}\subset C_c(\Omega; C_\#(Y^\bblm;\RRl)) \)
be a sequence converging to \(v\) in \(L^p(\Omega \times Y^\bblm;\RRd)\) and pointwise in \(\Omega\times Y^\bblm\). By (H3), we have, for all \(n,\,j\in\NN\),
\begin{equation*}
\begin{aligned}
f\Big( x, \frac{\bR x}{\e_n}, u_n(x)\Big) \geqslant f \Big(x, \frac{\bR x}{\e_n},
\psi_j\Big( x, \frac{\bR x}{\e_n} \Big) \Big) + \frac{\partial f}{\partial
\xi} \Big( x, \frac{\bR x }{\e}, \psi_j\Big( x, \frac{\bR x}{\e}\Big)\Big) \cdot\Big(
u_n(x) - \psi_j\Big( x, \frac{\bR x}{\e} \Big) \Big).
\end{aligned}
\end{equation*}
Integrating  this estimate over $\Omega$ and passing to the limit as $n\to \infty,$ we invoke Proposition \ref{prop:Radmissible} and  (H2)--(H3) to infer that 
\begin{equation}
    \begin{aligned}
     \liminf_{n\to\infty} F_{\e_n}(u_n) &=\liminf_{n\to\infty} \int_\Omega  f\Big( x, \frac{\bR x}{\e_n}, u_n(x)\Big)\,\dx \\&\geqslant \int_\Omega \fint_{Y^\bblm} f(x,y,\psi_j(x,y))\,\dx\dy  + \int_\Omega \fint_{Y^\bblm}  \frac{\partial f}{\partial \xi}\big(x,y,\psi_j(x,y))\cdot\big( v(x,y) - \psi_j(x,y)\big )\,\dx\dy
    \end{aligned}
\end{equation}
for all \(j\in\NN\). Letting \(j\to\infty\) in this inequality, Fatou's lemma and (H1) yield  
    \begin{align*}
    \liminf_{n\to\infty} F_{\e_n}(u_n) &\geqslant \int_\Omega \fint_{Y^\bblm
} f(x,y, v(x,y))\,\dy \dx = \int_\Omega \fint_{Y^\bblm } f(x,y, u(x) + v_1(x,y))\,\dy \dx\\ & \geq \inf_{w\in \sw}\int_\Omega \fint_{Y^\bblm}
f(x,y,u(x)
+ w(x,y))\,\dx\dy = \mathcal{F}_{\mathrm{hom}}(u).\qedhere
    \end{align*}
\end{proof}

\begin{proof}[Proof of Theorem~\ref{t.mt}]
Proving that both the \textit{\(\Gamma\)-liminf inequality}  
and the 
 \textit{recovery sequence} properties in  Theorem~\ref{t.mt} hold is equivalent to proving that (see \cite{Da93}) for all \(u\in\su\), we have
\begin{equation}
\label{eq:GammaConv}
\begin{aligned}
\mathcal{F}_{\mathrm{hom}}(u)=\Gamma\text{-}\liminf_{n\to\infty} F_{\epsi_n}(u) = \Gamma\text{-}\limsup_{n\to\infty} F_{\epsi_n}(u),
\end{aligned}
\end{equation}
where  \(\{\epsi_n\}_{n\in\NN}\subset
\RR^+\) is an arbitrary sequence
converging to \(0\) and
\begin{equation*}
\begin{aligned}
& \Gamma\text{-}\liminf_{n\to\infty} F_{\epsi_n}(u)
:=\inf \Big\{ \liminf_{n\to\infty}  F_{\epsi_n}(u_n)\!:\, u_n
\weakly  u \text{ in } L^p(\Omega;\RRd) \text{ as } n\to\infty,
\enspace \Aa u_n =0 \text{ for all }n\in\NN\Big\},\\
&\Gamma\text{-}\limsup_{n\to\infty} F_{\epsi_n}(u)
:=\inf \Big\{ \limsup_{n\to\infty}  F_{\epsi_n}(u_n)\!:\, u_n
\weakly  u \text{ in } L^p(\Omega;\RRd) \text{ as } n\to\infty,
\enspace \Aa u_n =0 \text{ for all }n\in\NN\Big\}.
\end{aligned}
\end{equation*}

Taking the infimum over all admissible sequences on \eqref{eq:lowerb}, we conclude from  Proposition~\ref{prop.lbnd} that
\begin{equation}\label{eq:GammaConv1}
\begin{aligned}
 \Gamma\text{-}\liminf_{n\to\infty} F_{\epsi_n}(u) \geq \mathcal{F}_{\mathrm{hom}}(u).
\end{aligned}
\end{equation}
On the other hand, 
 Proposition~\ref{prop:upperb} with \(\kappa=1\) and  \(\bar w_0 =0\) yields %
\begin{equation*}
\begin{aligned}
\Gamma\text{-}\limsup_{n\to\infty} F_{\epsi_n}(u)
\leqslant \int_\Omega \fint_{Y^\bblm} f(x, \kappa y,
u(x) +
\bar w_1(x,y))\,\dy \dx + \delta 
\end{aligned}
\end{equation*}
for all \(\delta>0\) and \(\bar w_1\in \sw \). Hence, taking the infimum over   \(\bar w_1\in \sw \), and then letting \(\delta\to0\), we get
\begin{equation}\label{eq:GammaConv2}
\begin{aligned}
\Gamma\text{-}\limsup_{n\to\infty} F_{\epsi_n}(u)
\leqslant \mathcal{F}_{\mathrm{hom}}(u).
\end{aligned}
\end{equation}
Because \(\Gamma\text{-}\liminf_{n\to\infty} F_{\epsi_n}(u) \leq \Gamma\text{-}\limsup_{n\to\infty} F_{\epsi_n}(u)\), we obtain \eqref{eq:GammaConv} from \eqref{eq:GammaConv1} and \eqref{eq:GammaConv2}.
\end{proof}

Next, we establish an integral representation for the functional \(\mathcal{F}_{\mathrm{hom}}\) introduced in Theorem~\ref{t.mt}. To prove this integral representation, we use the following measurable selection criterion, proved in \cite[Lemma~3.10]{FoKr10} (also see \cite{CaVa77}).

\begin{lemma}\label{lem:meassel}
Let \(Z\) be a separable metric space, let \(T\) be a measurable space, and let \(\Upsilon : T\to 2^Z\) be a multi-valued function such that (i) for every \(t\in T\), \(\Upsilon (t)\subset Z\) is nonempty and open, and (ii) for every \(z\in Z\), \(\{t\in T\!:\, z\in \Upsilon(t)\} \subset T\)  is measurable. Then, \(\Upsilon\) admits a measurable selection; that is, there exists a measurable function, \(\upsilon: T\to Z\), such that \(\upsilon(t) \in \Upsilon(t)\) for all \(t\in T\).
\end{lemma}

\begin{proposition}\label{prop:intrep}
Under the assumptions of
Theorem~\ref{t.mt}, for all \(u\in\su\), we have
\begin{equation}
\label{eq:intrep}
\begin{aligned}
\mathcal{F}_{\mathrm{hom}}(u) = \int_\Omega f_{\rm hom} (x,u(x))\,\dx,
\end{aligned}
\end{equation}
where
\begin{equation*}
\begin{aligned}
f_{\rm hom} (x,\xi)=\inf_{v\in\sv } \fint_{Y^\bblm} f(x,y,\xi + v(y))\,\dy
\end{aligned}
\end{equation*}
with \(\sv\) given by \eqref{eq:setVA}.
\end{proposition}

\begin{proof}
Let  \(u\in\su\). Note that, by (H2), we have

\begin{equation}
\label{eq:pgforfhom}
\begin{aligned}
0\leq f_{\rm hom} (x,\xi)\leq C(1+ |\xi|^p)
\end{aligned}
\end{equation}
for all \((x,\xi)\in\Omega\times \RRd\). Moreover,
\begin{equation}\label{eq:measfhom}
\begin{aligned}
x\in\Omega \mapsto  f_{\rm hom} (x,u(x))
\end{aligned}
\end{equation}
is a measurable map. In fact, let  \(V_\Aa\) be a countable and dense  subset of
\(\sv\)  with respect to the (strong) topology of \(L^p_\#(Y^\bblm;\RRd)\). We observe that such set \(V_\Aa\) exists because \(\sv\) is a subset of the separable metric space \(L^p_\#(Y^\bblm;\RRd)\).  Then, the continuity of \(f\) (see (H1)), Vitali--Lebesgue's theorem, and (H2) yield
\begin{equation*}
\begin{aligned}
\inf_{v\in\sv } \fint_{Y^\bblm} f(x,y,u(x) + v(y))\,\dy = 
\inf_{v\in V_\Aa} \fint_{Y^\bblm} f(x,y,u(x) + v(y))\,\dy,
\end{aligned}
\end{equation*}
from which we conclude the measurability of the map in \eqref{eq:measfhom}.

Fix \(w\in \sw\). For a.e.~\(x\in\Omega\), we have \(w(x,\cdot) \in \sv\); hence, for  a.e.~\(x\in\Omega\), %
\begin{equation*}
\begin{aligned}
\inf_{v\in\sv } \fint_{Y^\bblm} f(x,y,u(x) + v(y))\,\dy \leq \fint_{Y^\bblm} f(x,y,u(x) + w(x,y))\,\dy.
\end{aligned}
\end{equation*}
Integrating this estimate over \(\Omega\), and then taking the infimum over \(w\in \sw\), we conclude that
\begin{equation*}
\begin{aligned}
 \int_\Omega f_{\rm hom} (x,u(x))\,\dx \leq \mathcal{F}_{\mathrm{hom}}(u).
\end{aligned}
\end{equation*}

To prove the converse inequality, we first observe that, by \eqref{eq:pgforfhom}, we may assume that 
\begin{equation}\label{eq:fhomfinite}
\begin{aligned}
f_{\rm hom} (x,u(x)) \in \RR \text{  for all } x\in\Omega,
\end{aligned}
\end{equation}
without loss of generality. Fix  \(\delta>0\), and consider the multi-valued function  \(\Upsilon_\delta:\Omega \to 2^{L^p_\#(Y^\bblm;\RRd)}\) defined, for \(x\in\Omega\), by
\begin{equation*}
\begin{aligned}
\Upsilon_\delta(x) := \bigg\{v\in \sv\!:\, \fint_{Y^\bblm} f(x,y, u(x) + v(y))\,\dy <f_{\rm hom} (x,u(x)) + \delta\bigg\}. 
\end{aligned}
\end{equation*}
Also, let \(\bar\delta\in (0,\delta)\) be such that
\begin{equation}\label{eq:eqiiu}
\begin{aligned}
\int_E C(1+|u(x)|^p)\,\dx <\delta
\end{aligned}
\end{equation}
whenever \(E\subset\Omega\) is a measurable set with \(\mathcal{L}^\bbln(E)<\bar\delta\), where \(C\) is given by (H2).
 
By \eqref{eq:fhomfinite}, we have \(\Upsilon_\delta(x) \not= \emptyset\) for all \(x\in\Omega\). Furthermore, arguing as above, using the continuity of \(f\), Vitali--Lebesgue's theorem, and (H2), it can be checked that for each \(x\in\Omega\),  \(L^p_\#(Y^\bblm;\RRd)\setminus \Upsilon_\delta(x)\) is a closed subset of \(L^p_\#(Y^\bblm;\RRd)\). On the other hand, recalling the measurability of the map in \eqref{eq:measfhom}, we have that
\begin{equation*}
\begin{aligned}
x\mapsto h(x):=  \fint_{Y^\bblm} f(x,y, u(x) +
v(y))\,\dy -f_{\rm hom} (x,u(x)) - \delta
\end{aligned}
\end{equation*}
defines a measurable map for each \(v\in L^p_\#(Y^\bblm;\RRd)\). Thus, \(\{x\in\Omega\!:\, v\in \Upsilon_\delta(x)\} = h^{-1}((-\infty,0)) \) is a measurable set for each \(v\in L^p_\#(Y^\bblm;\RRd)\). Consequently, by Lemma~\ref{lem:meassel}, there exists a measurable selection, \(w_\delta:\Omega\to L^p_\#(Y^\bblm;\RRd) \), of \(\Upsilon_\delta\). Moreover, by Lusin's theorem, \(w_\delta \in L^p(\Omega_\delta;L^p_\#(Y^\bblm;\RRd)) \) for a suitable measurable set \(\Omega_\delta\) such that \(\mathcal{L}^\bbln(\Omega\setminus \Omega_\delta)<\bar\delta\). 

Finally, we define \(\bar w_\delta \in
\sw\) by setting \(\bar w_\delta(x) := w_\delta(x)\) if \(x\in\Omega_\delta\), and \(\bar w_\delta(x) := 0\)
if \(x\in\Omega\setminus \Omega_\delta\). Then, using the definition of \(\mathcal{F}_{\mathrm{hom}}(u)\), (H2),  \eqref{eq:eqiiu}, and \eqref{eq:pgforfhom},  we get
\begin{equation*}
\begin{aligned}
\mathcal{F}_{\mathrm{hom}}(u) &\leq \int_\Omega \fint_{Y^\bblm}
f(x,y,u(x)
+ \bar w_\delta(x,y))\,\dx\dy  \leq \delta + \int_{\Omega_\delta} \fint_{Y^\bblm}
f(x,y,u(x)
+  w_\delta(x,y))\,\dx\dy \\
&\leq \delta(1+\mathcal{L}^\bbln(\Omega)) + \int_\Omega f_{\rm hom} (x,u(x))\,\dx,
\end{aligned}
\end{equation*}
from which we conclude the desired inequality by letting \(\delta\to0\). 
\end{proof}

Finally, we prove Theorem~\ref{thm:main}.

\begin{proof}[Proof of Theorem~\ref{thm:main}]
Theorem~\ref{thm:main} is an immediate consequence of Theorem~\ref{t.mt} (also see \eqref{eq:GammaConv}) and Proposition~\ref{prop:intrep}.
\end{proof}

\section{The curl case} \label{sec:counting}

In this section, we prove Theorem~\ref{thm:alterchaRl} that provides
an equivalent alterative characterization
for the \(\bR\)-two-scale limit of bounded sequences in \(W^{1,p}\), corresponding to \(\Aa=\rm curl\) in Theorem~\ref{thm:main2}.
In this case, by Proposition~\ref{prop:compactenessR2sc}
and extracting a subsequence if necessary,
we have   \(u_\epsi\Rsc{\bR\text{-}2sc} u_0 \) and
 \(\grad u_\epsi\Rsc{\bR\text{-}2sc}  U_0 \) for some \(u_0\in
 L^p(\Omega \times Y^\bblm)\) and \(U_0\in
 L^p(\Omega \times Y^\bblm; \RRn)\).  Next, to  study   the relationship between \(u_0\) and \(U_0\), we  use Fourier analysis. As we mentioned before,
this is the approach adopted in \cite{BoGuZo10}; however,
 the arguments in \cite{BoGuZo10} hinge on the Parseval and
 Plancherel identities, which are  valid in \(L^2(\Omega)\) only.
 Instead, our main tool here relies on the following theorem, which 
may be found in \cite{Gr14}. For simplicity, we take \(Y^\bblm=[0,1)^\bblm\),
and we use  the Einstein
convention on repeated indices.

\begin{theorem}\label{thm:Grafakos} Let \( w \in L^p_\#(Y^\bblm)\), and define\footnote{In the literature, \(w_N\) are called the square partial sums of the Fourier series of \(w\).}
\begin{equation*}
\begin{aligned}
 w_N(y):=\sum_{k\in \ZZ^\bblm\atop |k|_\infty \leq N } \hat w_k e^{2\pi i k
\cdot y}, \enspace y\in Y^\bblm, \enspace N\in\NN,
\end{aligned}
\end{equation*}
where \(\hat w_k:= \int_{Y^\bblm} w(y)e^{-2\pi i k \cdot
y} \, \dy\), \(k\in \ZZ^\bblm \), are the Fourier coefficients of \(w\). Then, 
\begin{equation}\label{eq:FourierLpconv}
\begin{aligned}
&\Big\Vert \sup_{N\in\NN}| w_N| \Big\Vert_{L^p(Y^\bblm)} \leq C_{p,m} \Vert w\Vert_{L^p(Y^\bblm)}, \\ & \lim_{N\to\infty} \Vert
w_N -  w\Vert_{L^p(Y^\bblm)}=0, \\ &    
\lim_{N\to\infty}  w_N(y)
= w(y) \text{ for \aev\  } y\in Y^\bblm,
\end{aligned}
\end{equation}
where \(C_{p,m}\) is a positive constant  depending on \(p\) and \(m\) only.
Moreover, for all \(k\in\ZZ^\bblm\),
\begin{equation}
\label{eq:coefFourierbdd}
\begin{aligned}
 |\hat w_k|^p\leq \Vert w\Vert_{L^p(Y^\bblm)}^p.
\end{aligned}
\end{equation}
\end{theorem}

\begin{proof}
The proof of \eqref{eq:FourierLpconv} may be found in \cite[ Thm.~4.1.8 and Thm.~4.3.16]{Gr14}) (see also \cite[Def.~3.2.3]{Gr14}).
 
To prove \eqref{eq:coefFourierbdd}, we use Jensen's inequality
and the equality \(|e^{-2\pi i k \cdot
y}|=1\),   \(k\in\ZZ^\bblm\), to obtain
\begin{align*}
& |\hat w_k|^p\leq \int_{Y^\bblm} \big| w(y)e^{-2\pi i k \cdot
y}\big|^p \, \dy = \Vert w\Vert_{L^p(Y^\bblm)}^p. \qedhere
\end{align*}
\end{proof}

\begin{remark}\label{rmk:onGR2}
Let \(w\in L^p(\Omega;L^p_\#(Y^\bblm))\), and define
\begin{equation*}
\begin{aligned}
 v_N(x):=\int_{Y^\bblm}  |w_N(x,y) - w(x,y)|^p\,\dy, \enspace x\in\Omega, \enspace N\in\NN, 
\end{aligned}
\end{equation*}
where \(w_N(x,y):=\sum_{k\in \ZZ^\bblm\atop |k|_\infty \leq
N } \hat w_k (x)e^{2\pi
i k
\cdot y} \) with \(\hat w_k(x):= \int_{Y^\bblm} w(x,y)e^{-2\pi i k \cdot
y} \, \dy\), \(k\in \ZZ^\bblm \).
By \eqref{eq:FourierLpconv}, for \aev-\(x\in\Omega\), we have 
\begin{equation*}
\begin{aligned}
\sup_{N\in\NN}|v_N(\cdot)| \leq \tilde C_{p,m} \int_{Y^\bblm} |w(\cdot,y)|^p\,\dy \in L^1(\Omega) \,\text{ and } \lim_{N\to \infty} v_N(x) =0.
\end{aligned}
\end{equation*}
Thus, by the Lebesgue
dominated convergence theorem, it follows that \(w_N \to w\)
in \(L^p(\Omega\times Y^\bblm)\) as \(N\to\infty\).
\end{remark}

Next, we study some properties of the space \(\Gg_{\bR}^p\) introduced in \eqref{eq:spaceGR} that will be
useful in the sequel.
We first  observe that
if \(p=2\), then it can be checked that
\begin{equation*}
\begin{aligned}
\Gg_{\bR}^2= \bigg\{ w \in L^2_\#(Y^\bblm; \RRn)\!: \,   w(y)=\sum_{k\in
\ZZ^\bblm\backslash \{0\}} \lambda_k \bR^*k e^{2\pi
i k
\cdot y} \text{ for some } \{\lambda_k\}_{k\in \ZZ^\bblm\backslash \{0\}
} \subset \Cc \bigg\},
\end{aligned}
\end{equation*}
and we recover the space introduced in \cite{BoGuZo10}.

\begin{lemma}\label{lem:GRclosed}
Assume that \(\bR\) satisfies \eqref{eq:Rcriterion}. Then, the vector space \(\Gg_{\bR}^p\) introduced in \eqref{eq:spaceGR} is a closed subspace of \(L^p_\#(Y^\bblm; \RRn)\).
\end{lemma}

\begin{proof}
Let \(\{w_j\}_{j\in\Nn} \subset \Gg_{\bR}^p \) and \(w\in L^p_\#(Y^\bblm; \RRn)\) be such that \(\lim_{j\to\infty} \Vert w_j - w\Vert_{L^p(Y^\bblm;\RRn)} =0\). We want to show that \(w\in \Gg_{\bR}^p\).

For each \(j\in\NN\), let \(\hat w^j_k\) and \(\hat w_k\) denote the Fourier coefficients of \(w_j\) and \(w\), respectively. By \eqref{eq:coefFourierbdd}, we have \(\lim_{j\to\infty} |\hat w_k^j - \hat w_k| =0\).

On the other hand, by definition of \(\Gg_{\bR}^p\),  for each \(k\in\ZZ^\bblm\) and \(j\in\NN\), there exists \(\lambda_k^j\in\Cc\) such that \(\hat w_k^j = \lambda_k^j
\bR^*k\) and  \(\lambda_0^j =0\). In particular, \(\hat w_0 =0\). Fix \(k_0\in\ZZ^\bblm \backslash \{0\}\); by \eqref{eq:coefFourierbdd}, for all \(j,\, j'\in \NN\), we have
\begin{equation*}
\begin{aligned}
\big|\lambda_{k_0}^j - \lambda_{k_0}^{j'}\big| \big|\bR^*k_0\big| =\big |\big(\lambda_{k_0}^j - \lambda_{k_0}^{j'}\big)
\bR^*k_0\big| = \big|\hat w^j_{k_0} - \hat w^{j'}_{k_0}\big| \leq \Vert w_j - w_{j'} \Vert_{L^p(Y^\bblm:\RRn)}.
\end{aligned}
\end{equation*}
Because \(R^T(k_0) \not= 0\) by \eqref{eq:Rcriterion}, we conclude that 
\(\{\lambda_{k_0}^j\}_{j\in\NN}\) is a Cauchy sequence in \(\Cc\). Thus, there exists \(\lambda_{k_0} \in\Cc\) such that \(\lim_{j\to\infty} |\lambda_{k_0}^j - \lambda_{k_0}| =0 \). Consequently, passing the equality  \(\hat w_{k_0}^j = \lambda_{k_0}^j
\bR^*k_0\) to the limit as \(j\to\infty\), we obtain \(\hat w_{k_0} = 
\lambda_{k_0}
\bR^*k_0 \). 
\end{proof}

\begin{lemma}\label{lem:orthGR}
Assume that \(\bR\) satisfies \eqref{eq:Rcriterion}, and let \(w_0\in L^p_\#(Y^\bblm; \RRn)\) be such that
\begin{equation}
\label{eq:orthGR}
\begin{aligned}
\int_{Y^\bblm} w_0(y) \cdot \psi(y)\,\dy =0
\end{aligned}
\end{equation}
for all \(\psi\in C^\infty_\#(Y^\bblm;\RRn)\) with  \(\frac{\partial
\psi_l}{\partial y_\tau}\bR_{\tau l}=0\) in \( Y^\bblm\), where 
  \(l\in\{1, ...,\bbln\}\) and \(\tau \in\{1, ..., \bblm\} \). Then, \(w_0\in \Gg_{\bR}^p\).
\end{lemma}

\begin{proof} By contradiction, assume that \(w_0\not\in \Gg_{\bR}^p\). Then, using  Lemma~\ref{lem:GRclosed}, together with (a corollary
to) the Hahn--Banach theorem (see, for instance, \cite[Cor.~I.8]{Br83}), there exists \(v \in L^{p'}_\#(Y^\bblm;
\RRn)\)  such that
\begin{equation}
\label{eq:orthcd1}
\begin{aligned}
\int_{Y^\bblm}v(y) \cdot w(y)\,\dy =0
\end{aligned}
\end{equation}
 for all \(w\in \Gg_{\bR}^p\), and
\begin{equation}
\label{eq:orthcd2}
\begin{aligned}
\int_{Y^\bblm}v(y) \cdot w_0(y)\,\dy \not=0.
\end{aligned}
\end{equation}

We claim that \(\frac{\partial
v_l}{\partial y_\tau}\bR_{\tau l}=0\) in the sense of distributions. In fact, let \(\phi \in C^\infty_\#(Y^\bblm)\), and set \(w:=\bR^* \grad_y \phi\). Then, \(w \in L^p_\#(Y^\bblm; \RRn)\), \(\hat w_0 =0\), and \(\hat w_k = 2\pi i\hat \phi_k \bR^*k\) for all \(k\in\ZZ^\bblm \backslash\{0\}\). Thus, \(w\in \Gg_{\bR}^p\) and so, by \eqref{eq:orthcd1},
\begin{equation*}
\begin{aligned}
0 = \int_{Y^\bblm}v(y) \cdot w(y)\,\dy = \int_{Y^\bblm} \bR v(y)  \cdot \grad _y \phi(y)\,\dy,
\end{aligned}
\end{equation*}
which shows that \(0=\div_y (\bR v(y))=\div_y ( v(y)\bR^*) = \frac{\partial
v_l}{\partial y_\tau}\bR_{\tau l} \) in the sense of distributions because \(\phi \in C^\infty_\#(Y^\bblm)\) is arbitrary.

Using standard mollification techniques with a \(Y^\bblm\)-periodic, smooth  kernel, we may construct a sequence \(\{v_h\}_{h\in \NN} \subset 
C^\infty_\#(Y^\bblm;\RRn)\) such that  \(\div_y ( v_h(y)\bR^*)=0\) in \( Y^\bblm\) and \(\lim_{h\to\infty } \Vert v_h - v\Vert_{L^{p'}(Y^\bblm;\RRn)} =0\). Then, by \eqref{eq:orthGR} with \(\psi = v_h\) and Lebesgue dominated convergence theorem, we obtain \(\int_{Y^\bblm}w_0(y) \cdot v(y)\,\dy =0\), which contradicts \eqref{eq:orthcd2}. Thus, \(w_0\in \Gg_{\bR}^p\). \end{proof}

\begin{proposition}\label{prop:gradR2sc}
Let   \(\{u_\epsi\}_{\epsi} \subset
W^{1,p}(\Omega)\) be a bounded sequence, and assume that \(\bR\) satisfies \eqref{eq:Rcriterion}.
Then, there exist a subsequence \(\epsi'\preceq \epsi\) and   functions
\(u\in W^{1,p}(\Omega)\) and \(w\in L^p(\Omega; \Gg_{\bR}^p)\) such that
\begin{equation*}
\begin{aligned}
u_{\epsi'}\Rsc{\bR\text{-}2sc} u  \quad \text{and} \quad  
 \grad u_{\epsi'}\Rsc{\bR\text{-}2sc} \grad u + w.
\end{aligned}
\end{equation*}
\end{proposition}

\begin{proof} By the reflexibility of \(W^{1,p}(\Omega)\) and 
Proposition~\ref{prop:compactenessR2sc}, there exist \(u\in W^{1,p}(\Omega)\), \(u_0\in L^p(\Omega \times Y^\bblm)\), and \(U_0 \in L^p(\Omega\times Y^\bblm;\RRn)\) such that, extracting a subsequence if necessary, 
\begin{equation}\label{eq:bycompacteness}
\begin{aligned}
u_{\epsi} \weakly u \text{ weakly in } W^{1,p}(\Omega),\enspace  
u_{\epsi}\Rsc{\bR\text{-}2sc} u_0, \enspace   \text{and } \enspace   \grad u_{\epsi}\Rsc{\bR\text{-}2sc}
U_0.
\end{aligned}
\end{equation}

By the Rellich--Kondrachov theorem, \(u_\epsi \to u\) in \(L^p(\Omega)\). Hence,  Proposition~\ref{prop:R2scSc} and the uniqueness of the
\(\bR\)-two-scale limit (see Remark~\ref{rmk:uniq}) yield
\begin{equation*}
\begin{aligned}
u\equiv u_0.
\end{aligned}
\end{equation*}

We are left to prove that \(U_0\) in \eqref{eq:bycompacteness}
is of the form \(U_0 (x,y)= \grad u(x) + w(x,y)\) for some  \(w\in
L^p(\Omega; \Gg_{\bR})\).  Let  \(\Phi
\in C_c^\infty(\Omega;
C_\#^\infty(Y^\bblm;\RRn))\) be such that \(\frac{\partial
\Phi_l}{\partial y_\tau}\bR_{\tau l}=0\) in \(\Omega\times Y^\bblm\), 
  \(l\in\{1, ...,\bbln\}\), \(\tau\in\{1, ..., \bblm\} \). Then,
using integration by parts, the second
convergence in \eqref{eq:bycompacteness}, and  
the fact that   \(u_0\equiv
u\in W^{1,p}(\Omega)\), we obtain
\begin{equation*}
\begin{aligned}
&\int_{\Omega\times Y^\bblm} U_0(x,y) \cdot \Phi (x,y)\,\dx\dy \\&\quad=
\lim_{\epsi\to0^+} \int_\Omega \grad u_\epsi(x)\cdot  \Phi\Big(x,
\frac{\bR x}{\epsi}\Big)\,\dx\\
& \quad= -\lim_{\epsi\to0^+} \bigg(\int_\Omega u_\epsi(x) \div_x \Phi\Big(x,
\frac{\bR x}{\epsi}\Big)\,\dx + \frac1\epsi \int_\Omega u_\epsi(x)
\frac{\partial \Phi_l}{\partial y_\tau}\Big(x, 
\frac{\bR x}{\epsi}\Big)\bR_{\tau  l}\,\dx\bigg)\\
&\quad= -\lim_{\epsi\to0^+} \int_\Omega u_\epsi(x) \div_x \Phi\Big(x,
\frac{\bR x}{\epsi}\Big)\,\dx  =- \int_{\Omega\times Y^\bblm} u(x)
\div_x\Phi(x,y)\,\dx\dy   \\
&\quad= \int_{\Omega\times Y^\bblm} \grad u(x) \cdot \Phi (x,y)\,\dx\dy.
\end{aligned}
\end{equation*}
Hence,
\begin{equation*}
\begin{aligned}
\int_{\Omega\times Y^\bblm} \big(U_0(x,y) - \grad u(x)\big) \cdot
\Phi (x,y)\,\dx\dy =0
\end{aligned}
\end{equation*}
for all   \(\Phi
\in C_c^\infty(\Omega;
C_\#^\infty(Y^\bblm;\RRn))\)  such that \(\frac{\partial
\Phi_l}{\partial y_\tau}\bR_{\tau l}=0\) in \(\Omega\times Y^\bblm\).
Invoking Lemma~\ref{lem:orthGR}, we conclude for \aev-\(x\in\Omega\),
\(U_0(x,\cdot) - \grad u(x) \in \Gg_{\bR}^p\), which, together
with the fact that the map \((x,y) \mapsto U_0(x,y) - \grad u(x)\)
belongs to \(L^p(\Omega\times Y^\bblm;\RRn)\), concludes the proof.
\end{proof}

The next proposition shows that Proposition~\ref{prop:gradR2sc} fully
characterizes the \(\bR\)-two-scale limit of bounded sequences in \(W^{1,p}(\Omega)\). To the best of our knowledge,  in the
framework of \(\bR\)-two-scale convergence, this result is new in the literature even for \(p=2\).

\begin{proposition}\label{prop:reverseR2sc}
Let \(u\in W^{1,p}(\Omega)\) and \(w\in L^p(\Omega; \Gg_{\bR}^p)\), and assume that \(\bR\) satisfies
\eqref{eq:Rcriterion}. Then, there exists a bounded sequence   \(\{u_\epsi\}_{\epsi} \subset
W^{1,p}(\Omega)\)  such that
\begin{equation*}
\begin{aligned}
u_{\epsi}\Rsc{\bR\text{-}2sc} u  \quad \text{and} \quad  
 \grad u_{\epsi}\Rsc{\bR\text{-}2sc} \grad u + w.
\end{aligned}
\end{equation*}  
\end{proposition}

\begin{proof}

We first consider the case in which \(w\in W^{1,p}(\Omega; \Gg_{\bR}^p)\).

For each \(k\in \ZZ^\bblm\), define
\begin{equation*}
\begin{aligned}
\hat w_k(x):= \int_{Y^\bblm} w(x,y)e^{-2\pi i k \cdot y} \,\dy, \enspace x\in\Omega.
\end{aligned}
\end{equation*}
Because  \(w\in W^{1,p}(\Omega; \Gg_{\bR}^p)\), we have \(\hat w_k \in W^{1,p}(\Omega;\Cc^n)\) with \(\hat w_0\equiv0\); moreover, for \aev-\(x\in\Omega\), we have \(\hat w_k(x)=\lambda_k(x)\bR^*k\) for some \(\lambda_k\in W^{1,p}(\Omega;\Cc)\) with \(\lambda_0\equiv0\). 

For each \(N\in \NN\), let \(\tilde w_N \in W^{1,p}(\Omega; C_\#^\infty(Y^\bblm;\Cc))\) be the function defined by
\begin{equation*}
\begin{aligned}
\tilde w_N(x,y):= \sum_{k\in\ZZ^\bblm \atop |k|_\infty\leq N } \frac{1}{2\pi i}\lambda_k(x) e^{2\pi i k \cdot y}, \enspace (x,y) \in\Omega \times \RRm.
\end{aligned}
\end{equation*}
Note that  the function
\begin{equation*}
\begin{aligned}
 w_N(x,y):=\bR^* \grad_y \tilde w_N(x,y) = \sum_{k\in\ZZ^\bblm \atop |k|_\infty\leq N } \lambda_k(x) \bR^*k e^{2\pi i k \cdot y} = \sum_{k\in\ZZ^\bblm \atop |k|_\infty\leq N } \hat w_k(x) e^{2\pi i k \cdot y}, \enspace (x,y) \in\Omega
\times \RRm,
\end{aligned}
\end{equation*}
belongs to \( W^{1,p}(\Omega; C_\#^\infty(Y^\bblm;\RRn))\) and, by Remark~\ref{rmk:onGR2}, satisfies
\begin{equation}
\label{eq:convwN}
\begin{aligned}
\lim_{N\to\infty} \int_{\Omega\times Y^\bblm} |w_N(x,y)- w(x,y) |^p\,\dx\dy.
\end{aligned}
\end{equation}
Finally, for each \(\epsi\), we define \(\bar w_N:= {Re} (\tilde
w_N)\) and
\begin{equation*}
\begin{aligned}
u_{\epsi,N}(x):= u(x) + \epsi \bar w_N \Big(x, 
\frac{\bR x}{\epsi}\Big), \enspace x\in\Omega.
\end{aligned}
\end{equation*}
Then, \(u_{\epsi,N} \in W^{1,p}(\Omega)\) with
\begin{equation*}
\begin{aligned}
\Vert u_{\epsi,N} \Vert_{L^p(\Omega)} \leq \Vert u \Vert_{L^p(\Omega)} + \epsi \Vert\bar  w_N \Vert_{L^p(\Omega;C_\#(Y^\bblm))}
\end{aligned}
\end{equation*}
and
\begin{equation*}
\begin{aligned}
\Vert \grad u_{\epsi,N} \Vert_{L^p(\Omega;\RRn)} \leq \Vert\grad  u \Vert_{L^p(\Omega;\RRn)}
+ \epsi \Vert  \grad_x \bar w_N \Vert_{L^p(\Omega;C_\#(Y^\bblm;\RRn))} + \Vert  w_N \Vert_{L^p(\Omega;C_\#(Y^\bblm;\RRn))}.
\end{aligned}
\end{equation*}

Let \(\ffi \in L^{p'}(\Omega;C_\#(Y^\bblm))\) and 
\(\Phi \in L^{p'}(\Omega;C_\#(Y^\bblm;\RRn))\). 
By Proposition~\ref{prop:Radmissible}, we have
\begin{equation}
\label{eq:R2scuepsiN}
\begin{aligned}
\lim_{N\to\infty} \lim_{\epsi\to0^+} \int_\Omega u_{\epsi,N}(x) \ffi \Big(x, 
\frac{\bR x}{\epsi}\Big)\,\dx &= \lim_{N\to\infty} \lim_{\epsi\to0^+} \int_\Omega \Big( u(x) + \epsi \bar w_N \Big(x, 
\frac{\bR x}{\epsi}\Big) \Big)
\ffi \Big(x, 
\frac{\bR x}{\epsi}\Big)\,\dx\\ & = \int_{\Omega\times Y^\bblm}
 u(x) \ffi(x,y) \,\dx\dy
\end{aligned}
\end{equation}
and, also using  \eqref{eq:convwN},
\begin{equation}
\label{eq:R2scgraduepsiN}
\begin{aligned}
&\lim_{N\to\infty} \lim_{\epsi\to0^+} \int_\Omega \grad u_{\epsi,N}(x)
\cdot \Phi \Big(x, 
\frac{\bR x}{\epsi}\Big)\,\dx\\
&\quad =\lim_{N\to\infty} \lim_{\epsi\to0^+} \int_\Omega \Big(\grad u(x) +\epsi \grad _x \bar w \Big(x,  \frac{\bR x}{\epsi}\Big) + w_N  \Big(x,  \frac{\bR x}{\epsi}\Big) \Big)
\cdot \Phi \Big(x, 
\frac{\bR x}{\epsi}\Big)\,\dx\\
& \quad = \int_{\Omega\times Y^\bblm}
(\grad  u(x) + w(x,y)) \cdot \Phi(x,y) \,\dx\dy. 
\end{aligned}
\end{equation}

Due to the separability of \(L^{p'}(\Omega;C_\#(Y^\bblm))\) and \(L^{p'}(\Omega;C_\#(Y^\bblm;\RRn))  \)  and \eqref{eq:R2scuepsiN}--\eqref{eq:R2scgraduepsiN}, we can proceed as in \cite[proof of Prop.~1.11 (p.449)]{FeFo120} to find a sequence \(\{N_\epsi\}_\epsi\) such that \(N_\epsi\to\infty\) as \(\epsi \to 0^+\) and \(\tilde u_\epsi:= u_{\epsi, N_\epsi} \in W^{1,p}(\Omega) \)  satisfies
\begin{equation*}
\begin{aligned}
 \lim_{\epsi\to0^+} \int_\Omega \tilde u_\epsi(x)
\ffi \Big(x, 
\frac{\bR x}{\epsi}\Big)\,\dx = \int_{\Omega\times Y^\bblm}
 u(x) \ffi(x,y) \,\dx\dy
\end{aligned}
\end{equation*}
and 
\begin{equation*}
\begin{aligned}
 \lim_{\epsi\to0^+} \int_\Omega \grad\tilde u_\epsi(x)
\cdot \Phi \Big(x, 
\frac{\bR x}{\epsi}\Big)\,\dx = \int_{\Omega\times Y^\bblm}
(\grad  u(x) + w(x,y)) \cdot \Phi(x,y) \,\dx\dy 
\end{aligned}
\end{equation*}
for all  \(\ffi \in L^{p'}(\Omega;C_\#(Y^\bblm))\) and 
\(\Phi \in L^{p'}(\Omega;C_\#(Y^\bblm;\RRn))\); that is,
\begin{equation*}
\begin{aligned}
\tilde u_{\epsi}\Rsc{\bR\text{-}2sc} u  \quad \text{and} \quad  
 \grad \tilde u_{\epsi}\Rsc{\bR\text{-}2sc} \grad u + w.
\end{aligned}
\end{equation*}  
The boundedness of \(\tilde u_\epsi \) in \(W^{1,p}(\Omega)\) follows from Proposition~\ref{prop:R2sc-wc}.

To conclude, we treat the general  case in which \(w\in L^p(\Omega; \Gg_{\bR}^p)\). We claim that there exists a sequence \(\{\bar w_N\}_{N\in\NN} \subset W^{1,p}(\Omega;
\Gg_{\bR}^p)\) such that
\begin{equation}\label{eq:W1,papprox}
\begin{aligned}
\lim_{N\to\infty} \Vert \bar w_N - w \Vert_{L^p(\Omega\times Y^\bblm;\RRn)} = 0.
\end{aligned}
\end{equation}

Assume that the claim holds. Then, by the previous case, for each \(N\in\NN\), there exists a bounded sequence \(\{u_\epsi^N\}_\epsi \subset W^{1,p}(\Omega)\) such that
\begin{equation}\label{eq:R2scapprW1p}
\begin{aligned}
u_{\epsi}^N \Rsc{\bR\text{-}2sc} u  \quad \text{and} \quad  
 \grad u_{\epsi}^N\Rsc{\bR\text{-}2sc} \grad u + \bar w_N
\end{aligned}
\end{equation}
as \(\epsi\to0\).

Let \(\ffi \in L^{p'}(\Omega;C_\#(Y^\bblm))\) and 
\(\Phi \in L^{p'}(\Omega;C_\#(Y^\bblm;\RRn))\). Using \eqref{eq:R2scapprW1p} first, and then \eqref{eq:W1,papprox}, we obtain
\begin{equation*}
\begin{aligned}
\lim_{N\to\infty} \lim_{\epsi\to0^+} \int_\Omega u_{\epsi}^N(x)
\ffi \Big(x, 
\frac{\bR x}{\epsi}\Big)\,\dx = \int_{\Omega\times Y^\bblm}
 u(x) \ffi(x,y) \,\dx\dy
\end{aligned}
\end{equation*}
and 
\begin{equation*}
\begin{aligned}
\lim_{N\to\infty} \lim_{\epsi\to0^+} \int_\Omega \grad u_{\epsi}^N(x)
\cdot \Phi \Big(x, 
\frac{\bR x}{\epsi}\Big)\,\dx &=\lim_{N\to\infty} \int_{\Omega\times Y^\bblm}
(\grad  u(x) + \bar w_N(x,y)) \cdot \Phi(x,y) \,\dx\dy\\
& = \int_{\Omega\times Y^\bblm}
(\grad  u(x) + w(x,y)) \cdot \Phi(x,y) \,\dx\dy.
\end{aligned}
\end{equation*}
Finally, arguing as in the previous case, we can find a 
 sequence \(\{N_\epsi\}_\epsi\) such that \(N_\epsi\to\infty\)
as \(\epsi \to 0^+\) and \(\tilde u_\epsi:= u_{\epsi}^{N_\epsi}
\in W^{1,p}(\Omega) \)  satisfies the requirements.

We are left to prove \eqref{eq:W1,papprox}. As before, for each \(k\in \ZZ^\bblm\), define
\begin{equation*}
\begin{aligned}
\hat w_k(x):= \int_{Y^\bblm} w(x,y)e^{-2\pi i k \cdot y} \,\dy, \enspace
x\in\Omega.
\end{aligned}
\end{equation*}
Because  \(w\in L^p(\Omega; \Gg_{\bR}^p)\), we have \(\hat
w_k \in L^p(\Omega;\Cc^\bbln)\) with \(\hat w_0\equiv0\); moreover,
for \aev-\(x\in\Omega\), we have \(\hat w_k(x)=\lambda_k(x)\bR^*k\)
for some \(\lambda_k\in L^p(\Omega;\Cc)\) with \(\lambda_0\equiv0\).
Then, for each \(k\in\ZZ^\bblm\), we can find a sequence 
\(\{\lambda_k^j\}_{j\in\NN} \subset W^{1,p}(\Omega;\Cc)\), with \(\lambda_0^j =0 \), such that \(\lambda_k^j \to \lambda_k\) in \(L^p(\Omega;\Cc)\) as \(j\to \infty\). In particular, we have
\begin{equation}\label{eq:appoxlambdas}
\begin{aligned}
\lim_{j\to\infty}\bigg(\int_\Omega |\lambda_k^j (x)\bR^*k - \lambda_k (x)\bR^*k|^p\,\dx\bigg)^{\frac{1}{p}}= \lim_{j\to\infty}|\bR^*k|\bigg(\int_\Omega |\lambda_k^j  (x)- \lambda_k
(x)|^p\,\dx\bigg)^{\frac{1}{p}}=0.
\end{aligned}
\end{equation}
Fix \(N\in\NN\); by \eqref{eq:appoxlambdas}, there exits \(j_N\in\NN\) such that
\begin{equation}\label{eq:goodjn}
\begin{aligned}
\sum_{k\in\ZZ^\bblm \atop |k|_\infty \leq N} \bigg(\int_\Omega |\lambda_k^{j_N} (x)\bR^*k - \lambda_k
(x)\bR^*k|^p\,\dx\bigg)^{\frac{1}{p}} \leq \frac{1}{N }.
\end{aligned}
\end{equation}
Defining
\begin{equation*}
\begin{aligned}
\tilde w_N(x,y):= \sum_{k\in\ZZ^\bblm \atop |k|_\infty \leq N} \lambda_k^{j_N}(x)\bR^*ke^{2\pi ik \cdot y} \enspace \text{ and } \enspace \bar w_N := Re(\tilde w_N),
\end{aligned}
\end{equation*}
we have \(\bar w_N \in W^{1,p}(\Omega;
\Gg_{\bR}^p)\); moreover, invoking Remark~\ref{rmk:onGR2} and \eqref{eq:goodjn}, we have
\begin{equation*}
\begin{aligned}
&\limsup_{N\to\infty} \bigg( \int_{\Omega\times Y^\bblm} |\bar  w_N(x,y) - w(x,y)|^p\,\dx\dy\bigg)^{\frac{1}{p}} \leq\limsup_{N\to\infty} \bigg( \int_{\Omega\times Y^\bblm} |\tilde  w_N(x,y)
- w(x,y)|^p\,\dx\dy\bigg)^{\frac{1}{p}} \\
&\quad\leq   \limsup_{N\to\infty}\bigg[ \bigg(\int_{\Omega\times Y^\bblm} |\tilde w_N(x,y) - w_N(x,y)|^p\,\dx\dy\bigg)^{\frac{1}{p}} +\bigg( \int_{\Omega\times
Y^\bblm} | w_N(x,y) - w(x,y)|^p\,\dx\dy\bigg)^{\frac{1}{p}} \bigg] \\
&\quad =   \limsup_{N\to\infty} \bigg(\int_{\Omega\times
Y^\bblm} \bigg|\sum_{k\in\ZZ^\bblm \atop |k|_\infty \leq N}(\lambda_k^{j_N}
(x) - \lambda_k
(x))\bR^*k e^{2\pi ik \cdot y}\bigg|^p\,\dx\dy\bigg)^{\frac{1}{p}}
\leq    \limsup_{N\to\infty} \frac1N =0,  
\end{aligned}
\end{equation*}
which concludes the proof of \eqref{eq:W1,papprox}.  
\end{proof}

\begin{remark}\label{rmk:relGRAR*} Let \(\Omega\subset \RRn\) be a simply connected, bounded, and open set. Applying  Proposition~\ref{prop:reverseR2sc} to \(u=0\) and \(w\in \Gg_{\bR}^p\), we can find a bounded sequence
  \(\{u_\epsi\}_{\epsi} \subset
W^{1,p}(\Omega)\)  such that
\begin{equation*}
\begin{aligned}
u_{\epsi}\Rsc{\bR\text{-}2sc} 0  \quad \text{and} \quad  
 \grad u_{\epsi}\Rsc{\bR\text{-}2sc}  w.
\end{aligned}
\end{equation*} 
Then, by Proposition~\ref{prop:R2sc-wc}, we have  \(u_\epsi\weakly
0\) in \(W^{1,p}(\Omega)\) and \(\int_{Y^\bblm}
w(y)\,\dy =0\). On the other hand, using the uniqueness of the
\(\bR\)-two-scale limit (see Remark~\ref{rmk:uniq}) and 
Proposition~\ref{prop:necR2sc} with \(\bbld=\bbln\)
and \(\Aa=\rm curl\) in \(\RRn\), we conclude that that \(w\in  L^p_\#(Y^\bblm;\RRn)\) is \(\A_{\bR^*}\)-free in the sense of Definition~\ref{def:AR*free}. 

Conversely, if  \(w\in
 L^p_\#(Y^\bblm;\RRn)\) is \(\A_{\bR^*}\)-free with  \(\int_{Y^\bblm}
w(y)\,\dy =0\), then by Proposition~\ref{prop:sufR2sc} there exists a  bounded and
\(\Aa \)-free sequence, \(\{u_\epsi\}_\epsi\), 
in \(L^p(\Omega;\RRn)\) such
that \(u_{\epsi}\Rsc{\bR\text{-}2sc}
u \). As we are in the \(\Aa=\rm curl\) case and \(\Omega\) is simply connected, we can find a bounded
sequence, \(\{v_\epsi\}_\epsi\), in \(W^{1,p}(\Omega)\) such
that \(\int_\Omega v_\epsi(x)\,\dx = 0\) and  \(\grad v_\epsi = u_\epsi\). Then, by Proposition~\ref{prop:R2sc-wc}, we deduce that \(v_\epsi\weakly
0\) in \(W^{1,p}(\Omega)\). Finally, 
Remark~\ref{rmk:uniq} and Proposition~\ref{prop:gradR2sc} yield \(w\in \Gg_{\bR}^p\). 

Thus, in the \(\Aa=\rm curl\) case in \(\RRn\), we have that \(w\in
 L^p_\#(Y^\bblm;\RRn)\) is \(\A_{\bR^*}\)-free in the sense of
Definition~\ref{def:AR*free} if and only if  \(w\in \Gg_{\bR}^p\).
\end{remark}

To conclude, we observe that Theorem~\ref{thm:alterchaRl} is
an immediate consequence of the previous results.

\begin{proof}[Proof of Theorem~\ref{thm:alterchaRl}]
The claim in Theorem~\ref{thm:alterchaRl} follows from Propositions~\ref{prop:gradR2sc}
and \ref{prop:reverseR2sc}.
\end{proof}

\section*{Acknowledgements}
I. Fonseca and R.  Venkatraman acknowledge the Center for Nonlinear Analysis where part of this work was carried out. The research of I. Fonseca was partially funded by the National Science Foundation under Grants No. DMS-1411646 and DMS-1906238. The research of R.  Venkatraman was funded by the National Science Foundation under Grant No. DMS-1411646.

\bibliographystyle{plain}

\end{document}